\documentclass[12pt,a4paper]{amsart}
\usepackage[english]{babel}
\usepackage{amsmath}
\usepackage{amsthm}
\usepackage{amssymb}
\usepackage{amsfonts}
\usepackage{enumitem}
\usepackage{dsfont}
\usepackage[utf8]{inputenc}
\usepackage[a4paper, top=3cm, bottom=3cm, left=2cm, right=2cm,marginparwidth=1.75cm]{geometry}
\usepackage{url}
\usepackage{mathrsfs,yfonts}
\usepackage{bbm}
\usepackage[normalem]{ulem}
\usepackage{thmtools}
\usepackage{thm-restate}
\usepackage[toc,page]{appendix}
\usepackage{listings}
\usepackage{comment, cancel}
\usepackage{multirow}
\usepackage{xcolor,soul}
\usepackage[hypertexnames=false]{hyperref}
\usepackage{cleveref}

\newtheorem{theorem}{Theorem}[section]
\newtheorem{lemma}[theorem]{Lemma}

\newtheorem{prop}[theorem]{Proposition}
\newtheorem{thm}[theorem]{Theorem}
\newtheorem{cor}[theorem]{Corollary}

\newtheorem*{cor*}{Corollary}
\newtheorem*{prop*}{Proposition}
\theoremstyle{definition}

\newtheorem{remark}[theorem]{Remark}
\newtheorem{definition}[theorem]{Definition}

\newcommand{\mc}{\mathcal}
\newcommand{\E}{\ensuremath{\mathds{E}}}

\newcommand{\R}{\ensuremath{\mathbb{R}}}

\newcommand{\va}{\ensuremath{\gamma}}
\newcommand{\vb}{\ensuremath{\nu}}
\newcommand{\vc}{\ensuremath{q}}

\newcommand{\vf}{\ensuremath{\phi}}

\def\pcut{{{}_\square}}

\usepackage{amsaddr}
\usepackage[final]{pdfpages}

\title{The cut norm and Sampling Lemmas for unbounded kernels}
\author{Panna T\'imea Fekete}
\address{Institute of Mathematics, E\"otv\"os Lor\'and University\\POB 120, H-1518 Budapest, Hungary\\HUN-REN Alfr\'ed R\'enyi Institute of Mathematics\\POB 127, H-1364 Budapest, Hungary}
\email{feketept@renyi.hu}
\author{D\'avid Kunszenti-Kov\'acs}
\address{HUN-REN Alfr\'ed R\'enyi Institute of Mathematics\\POB 127, H-1364 Budapest, Hungary}
\email{daku@renyi.hu}

\begin{document}
\maketitle
\tableofcontents
\begin{abstract} Generalizing the bounded kernel results of Borgs, Chayes, Lovász, Sós and Vesztergombi (\cite{Borgs}), we prove two Sampling Lemmas for unbounded kernels with respect to the cut norm. On the one hand, we show that given 
a (symmetric) kernel $U\in L^p([0,1]^2)$ for some $3<p<\infty$, the cut norm of a random $k$-sample of $U$ is with high probability within $O(k^{-\frac14+\frac{1}{4p}})$ of the cut norm of $U$. The cut norm of the sample has a strong bias to being larger than the original, allowing us to actually obtain a stronger high probability bound of order $O(k^{-\frac 12+\frac1p+\varepsilon})$ for how much smaller it can be (for any $p>2$ here). These results are then partially extended to the case of vector valued kernels.

On the other hand, we show that with high probability, the $k$-samples are also close to $U$ in the cut metric, albeit with a weaker bound of order $O((\ln k)^{-\frac12+\frac1{2p}})$ (for any appropriate $p>2$). As a corollary, we obtain that whenever $U\in L^p$ with $p>4$, the $k$-samples converge almost surely to $U$ in the cut metric as $k\to\infty$.
\end{abstract}

%%%%%%%%%%%%%%%%%%%
%
%
\section{Introduction}
%
%
%%%%%%%%%%%%%%%%%%%
The cut norm for matrices was introduced by Frieze and Kannan in \cite{FK}, and used to formulate a regularity lemma that provides a weaker regularity notion than Szemerédi's, but with better bounds and the advantage of allowing for efficient algorithms. The main feature of the cut norm is that it emphasizes structure over randomness, and ``filters out noise'' via averaging in the following sense. For an $n\times n$ matrix with all entries 1, the cut norm equals 1, whereas for a matrix with iid $\pm 1$ entries with 0 expectation, the cut norm is with high probability of order $O(n^{-1/2})$. In other words, if two matrices differ by essentially a white noise, then their difference has very small cut norm, but as the difference becomes more ``structured'', its norm increases. This norm has played a key role in the development of (dense) graph limit theory due to its strong combinatorial ties.

In graph limit theory, graphs are up to isomorphism identified with their adjacency matrix, and it turns out that under the notion of homomorphism density convergence, the limiting objects can be represented by graphons, which are symmetric, measurable $[0,1]^2\to[0,1]$ functions. These functions can essentially be viewed as continuous analogues of adjacency matrices of graphs. Leaving the realm of simple graphs towards weighted graphs and multigraphs, the next natural classes of limit objects are represented by an unbounded, $L^p$ kernel for some $p<\infty$, see for instance the cut distance focused papers \cite{BCCZ1, BCCZ2} by Borgs, Chayes, Cohn and Zhao, and then vector valued graphs, see \cite{KKLSz1} by Lovász, Szegedy, and the second author.

The First Sampling Lemma as originally proved by Alon, de la Vega, Kannan and Karpinski in \cite{Alon} is concerned with estimating the distance of two matrices (or, more generally, $r$-dimensional arrays) in the cut norm by sampling. That result was later improved and extended to graphons and bounded kernels -- symmetric functions in $L^\infty([0,1]^2)$ -- by Borgs, Chayes, Lovász, Sós and Vesztergombi in \cite{Borgs}, proving an error bound for $k$-sampling of order $O(k^{-1/4})$ with high probability.

If one is interested in whether the sample is actually close to the original kernel, not just at roughly the same distance from the origin, the norm itself does not quite cut it, however. Namely, for simple graphs, the result of sampling does not depend on the labelling of the vertices, only on the isomorphism class of the graph; similarly, the distribution of the random samples of a kernel does not change if we apply a measure preserving transformation to the underlying unit interval $[0,1]$ to ``permute'' the kernel $U$ (cf. the effect of relabelling the vertices of a finite graph on its adjacency matrix). Hence we cannot expect norm concentration of the samples around $U$ unless it is invariant under such transformations, i.e., is constant. This is where the so-called \emph{cut metric} comes in, which essentially factorizes along these isomorphism classes. It was shown in the above mentioned paper \cite{Borgs} that for bounded kernels, the cut distance of a sample from the original kernel is with high probability of order $O((\ln k)^{-1/2})$. The significantly worse order of magnitude in this Second Sampling Lemma compared to the First Sampling Lemma is due to having to make use of the weak Szemerédi Regularity Lemma.
These sampling lemmas are a key ingredient in showing that the cut metric for matrices (or bounded kernels) is equivalent to the metric induced by left-convergence (also called density convergence) in the context of dense graph limit theory.

The second author showed in \cite{KKD} that for unbounded kernels, density convergence and convergence in cut distance are not equivalent. Indeed, it is known that for unbounded kernels $W$, the random $k$-samples left-converge with probability 1 to $W$ as $k\to \infty$ (Theorem 3.8  in \cite{KKLSz1} cited above), but this question was open in the cut distance. In particular, this indicates that one may not automatically expect the second sampling lemma to extend to unbounded kernels.

Under mild conditions, we derive a First Sampling Lemma for unbounded kernels. We provide bounds on the difference between the cut norm of an unbounded kernel and the cut norm of a sample with high probability. The upper and lower bounds turn out to be of very different orders of magnitude, with both depending on the $L^p$ class the kernel is in. The bounds get arbitrarily close in order to the bound for $L^\infty$ kernels as $p\to\infty$, at the cost of a larger, but still polynomially small (as opposed to exponentially small) exceptional event set. We note that such a polynomial sized exceptional set is still adequate for most applications.

We also prove a Second Sampling Lemma with similar conditions, and derive as a corollary that for any kernel in $L^p$ ($p>4$), the samples converge almost surely in the cut metric to the original kernel.

%%%%%%%%%%%%%%%%%%%
%
%
\section{Main results}\label{Section:Theorems}
%
%
%%%%%%%%%%%%%%%%%%%

Throughout the paper the sequence of random variables $X_0,X_1,\ldots$ will be assumed to consist of i.i.d. random variables with uniform distribution in the interval $[0,1]$, and $k$ will be an integer parameter. Initially, $U$ will denote a symmetric, real valued measurable function $[0,1]^2\to\R$, but in the last part of the paper we will be looking at vector valued kernels $U$. Let $X$ denote the random $k$-vector $X=(X_1,X_2,\ldots,X_k)$, and let $\mc{U}_X:=\mc{U}[X]$ denote the step-function on $[0,1]^2$ with uniform steps of length $1/k$ in each variable, and values given by $(U(X_i,X_j))_{i\neq j\in[k]}$, and 0 on the main diagonal (in order to avoid having to use values concentrated on a null-set).

Recall that given an $n\times n$ matrix $\mc{A}$ and a function $U\in L^1([0,1]^2)$, their \emph{cut norm} is defined as:
\[\|\mc{A}\|_\pcut = \frac{1}{n^2} \max_{S,T \subseteq[n]} \left|\sum_{i\in S,j\in T} \mc{A}_{ij}\right| \quad \text{ and } \quad \|U\|_\pcut = \sup_{S,T \subseteq[0,1]} \left| \int_{S\times T} U(x, y) dx dy\right|,\] respectively.

For two functions $U_1,U_2\in L^1([0,1]^2)$, their \emph{cut distance} is defined as follows.
 Let $\Psi$ denote the set of measure preserving bijections $[0,1]\to[0,1]$. Then
 \[ \delta_\pcut(U_1,U_2):=\inf_{\psi\in\Psi}\|U_1-U_2\circ\psi\|_\pcut,
 \]
where $(U_2\circ\psi)(x,y):=U_2(\psi(x),\psi(y))$.

Our goal is to provide a high probability bound on the difference of the norm
$
\big\|\mc{U}_X\big\|_\pcut$ of the random sample and the original norm $\big\|U\big\|_\pcut
$. The typical application would be to check with high probability that $\|U\|_\pcut$ is small, via looking at the samples $\mc{U}_X$. Indeed, in the graph limit theory context, the cut norm usually appears in various bounds (e.g., in Counting Lemmas or the weak version of Szemerédi's Regularity lemma by Frieze and Kannan) in the form $\|U_1-U_2\|_\pcut$, where it is sufficient to test for its smallness for applicability. On the other hand, we also want to obtain a concentration result for the samples around the kernel $U$, i.e., bound with high probability the cut distance $\delta_\pcut(\mc{U}_X, U)$.

Our main results are generalizations of the graphon sampling lemmas (\cite[Theorems 4.6 (i) and 4.7 (i)]{Borgs}) to the unbounded real valued case, and for the First Samping Lemma, partially to the vector valued case (the upper bound is valid only in finite dimensions). For the sake of readability, we here eliminate a parameter and formulate a simplified version of our real valued result, proved in full in Section \ref{section:proof_main}.

\begin{thm}{(First Sampling Lemma for Unbounded Kernels).}\label{Thm:First_sampling} 
Let $k\geq 2$ be an integer, $p>2$, and $U\in L_{sym}^p([0,1]^2)$. 
\begin{enumerate}
\item For any $\vf>0$,
the upper bound
\[
\|\mc{U}_X\|_\pcut-\|U\|_\pcut\leq \left(30\|U\|_p+
6\sqrt{2\vf p}
      (\|U\|_1+2\|U\|_p)
      \frac{\sqrt{\ln k}}{k^{\frac{p-3}{4p}-\vf}}\right)
      k^{-\frac14+\frac1{4p}}
\]
holds with probability at least
$1-4k^{-\vf p}$.
\item For any $\va\in(1/p,1/2)$, the lower bound
\[
\|\mc{U}_X\|_\pcut-\|U\|_\pcut\geq -\left(\|U\|_\pcut+6\sqrt{2(\va p-1)}(\|U\|_1+2\|U\|_p)\right) k^{-1/2+\va}\sqrt{\ln k}
\]
holds with probability at least $1- 4k^{1-\va p}$.
\end{enumerate}
\end{thm}

\begin{remark}
In the upper bound, the second term within the parentheses is $o(1)$ when $\vf\in \left(0, \frac{p-3}{4p}\right)$, and then the true order of magnitude is $k^{-\frac14+\frac1{4p}}$. This is only possible for $p>3$, however even in the $p\in(2,3]$ regime the total order of magnitude is $O(\sqrt{\ln k} k^{-\frac12+\frac1p+\vf})$, and $\vf$ can be chosen small enough to make this $o(1)$.

In the lower bound, the term $\|U\|_\pcut$ on the RHS seems out of place. However, it is trivially bounded above by $\|U\|_1$, so this lower bound can easily be made dependent only on $L^p$ norms. Also, the reason for that term existing is actually a question of normalization for the cut norm of the samples that we will address in detail in Section \ref{sect:lower}.

We also note that for any kernel $U\in\bigcap_{p\in[1,\infty)} L^p_{sym}([0,1]^2)$ (essentially the class of kernels for which all homomorphism densities are finite, see \cite{KKLSz1}), the order of magnitude of the upper bound can be made arbitrarily close to the $k^{-1/4}$ seen in the bounded case.
\end{remark}

This First Sampling Lemma then allows us to derive the following generalization of the Second Sampling Lemma to unbounded kernels. Note that we do not bound the norm of the difference itself, but rather the cut distance of the two functions. 

\begin{thm}{(Second Sampling Lemma for Unbounded Kernels).}\label{Thm:Second_sampling} 
Let $k\geq 2$ be an integer, $p>2$, and $U\in L_{sym}^p([0,1]^2)$.
For any $\vf\in\left(0, \frac{1}{2}- \frac{1}{p}\right)$ we have that with probability at least
$1-e^{\frac{-k^2}{2\log_2 k}}-4k^{-\vf p}$,
\[
\delta_\pcut(\mc{U}_X,U)\leq \frac{C}{(\ln k)^{\frac12-\frac1{2p}}},
\]
where $C$ depends only on $p$, $\|U\|_1$, $\|U\|_p$ and $\vf$.
\end{thm}

\begin{remark} The reason why the cut norm is not the right thing to look at here is as follows. If we replace $U$ by some $U\circ\psi$ ($\psi\in\Psi$) -- amounting to a relabelling of vertices at the level of graphs --, then the distribution of the random stepfunctions $\mc{U}_X$ will be the same as that of $\mc{U\circ\psi}_X$, so we cannot expect a better bound than the constant $\frac12\sup_{\psi\in\Psi}\|U-U\circ\psi\|_\pcut$.

Also, note again that $p\to\infty$ leads to the order $O(1/\sqrt{\ln k})$ valid for bounded kernels (\cite[Lemma 10.16]{Lovasz}).
\end{remark}

As a corollary, we obtain the following almost sure convergence result for the samples.

\begin{cor}\label{cor:almost_sure_left}
Let $p>4$, and for each integer $k\geq2$, let $X^{(k)}\in[0,1]^k$ be a uniform random $k$-vector. Then for any kernel $U\in L^p_{sym}([0,1]^2)$, the convergence $\delta_\pcut(U,\mc{U}_{X^{(k)}})\to0$ holds with probability 1.
\end{cor}

In our proof of the First Sampling Lemma, we shall proceed by establishing intermediate steps, splitting the error into the deterministic, "systematic error" $\E_X\left[ \|\mc{U}_X\|_\pcut \right] -\|U\|_\pcut$ (i.e., comparing the expected value of the random norm to the original), and the random "dispersion" term $\E_X\left[ \|\mc{U}_X\|_\pcut \right] -\|\mc{U}_X\|_\pcut$ (i.e., exploiting concentration inequalities to bound with high probability the error between the expectation and the true random norm).

First, we deal with the ``dispersion'' term (Section \ref{Section:Dispersion}), which is responsible for the  polynomial bound on the size of the exceptional set. Then we proceed to prove a lower bound on the ``systematic error'' (Section \ref{sect:lower}), matching the results obtained in the bounded kernel case.
Finally, we provide an upper bound on the ``systematic error'' (Section \ref{Section:upper_bound}), which is of a significantly larger order of magnitude, and asymptotically matches the bounded kernel case as $p\to\infty$.

The final steps of the proof will be provided in Section \ref{section:proof_main}.
The upper bound will be shown in Theorem \ref{thm:upper_bound}, and the lower bound in Theorem \ref{thm:lower_bound}.

In Section \ref{Section:Banach}, we then turn our attention to the First Sampling Lemma for vector valued kernels. These arise naturally in the context of coloured graph limits or limits of multigraphs, see for instance \cite{KKLSz1}. Whilst being able to maintain the lower bound of the real valued case, the upper bound behaves very differently, and we are only able to obtain a finite upper bound in finite dimensions. We actually suspect that in the infinite dimensional setting, one may be able to construct an unbounded vector valued kernel for which the norms of the samples exhibit very weak or no concentration around their expectation, even under $L^p$ constraints.

In Section \ref{sect:second}, we use our results from Section \ref{section:proof_main} to derive the generalisation of the Second Sampling lemma, i.e., we show a high probability upper bound for $\delta_\pcut(\mc{U}_X, U)$.

Finally, we note that in Subsection \ref{subsect:original}, we sketch an approach to the First Sampling Lemma for unbounded kernels based on the original sub-sampling approach of \cite{Alon}. This approach leads to interesting intermediate results that are not covered by the main approach, however the final bounds obtained are slightly worse than in Theorem \ref{Thm:First_sampling}. The full technical details have therefore been relegated to the Appendix (Section \ref{sect:Appendix}).

%%%%%%%%%%%%%%%%%%%
%
%
\section{Preliminaries} \label{Section:Preliminaries}
%
%
%%%%%%%%%%%%%%%%%%%
Throughout this paper, given an integer $a\geq 1$, we shall let $[a]$ denote the set of integers $\left\{1,\ldots,a\right\}$.
For an index set $I$ and a vector $Z\in[0,1]^I$, let \(Z_A\) denote the restriction of \(Z\) to the indices in the set \(A\subset I\) (i.e., \(Z_A\) is a vector of length \(|A|\), which we obtain by keeping the coordinates in the set \(A\) and deleting all other coordinates in \(Z\)).

We shall be making use of the following generalised version of Azuma’s Inequality.
\begin{lemma}\label{Lemma:generalised_azuma}
Let \((\Omega, \mc{A}, \pi)\) be a probability space and \(X\) be a random point of \(\Omega^n\) (chosen according to the product measure). Let \(f_1, \dots, f_\ell : \Omega^n \rightarrow \R\) be measurable and integrable functions, such that for any \(1\leq a\leq n\) the event sets
\[H_{\alpha_a,a,i} :=\Bigg\{\bigg | \E_{X}\left[f_i(X)\Big|X_{[a]}\right]-\E_{X}\left[f_i(X) \Big| X_{[a-1]}\right] \bigg|\geq \alpha_a\Bigg\}\]
for all \(1\leq i\leq \ell\) are contained in a set \(H_{\alpha_a,a}\). Then
\begin{equation*}
    \begin{split}
        \mathds{P}_{X}\left( \Big|f_i(X) - \E_{X}[f_i(X)] \Big| \geq 2\mu \sqrt{\sum_{a=1}^n \alpha_a^2} \quad \text{for some } i\in [\ell]\right) \leq 2 \ell  e^{-2\mu^2} + \mathds{P}_{X}\left(\bigcup_{a \in [n]}H_{\alpha_a,a}\right)
    \end{split}
\end{equation*}
for any \(\mu, \alpha_a \geq 0\).

More precisely, there exists a set $S_\mu$ with $\pi^n(S_\mu)\leq 2 \ell e^{-2\mu^2}$ such that for almost all $X\in\Omega^n\setminus\left(S_\mu\bigcup\cup_{a \in [n]}H_{\alpha_a,a}\right)$, we have
\[
\Big|f_i(X) - \E_{X}[f_i(X)] \Big| < \mu \sqrt{\sum_{a=1}^n \alpha_a^2} \quad \forall i\in [\ell].
\]
\end{lemma}
The proof of this claim goes similarly to the one in \cite[Prop. 34]{TV}.

Another useful statement is a generalisation of the well known Chebyshev's inequality, which we will use in the following form.
\begin{lemma}\label{Lemma:generalised_Chebyshev}
For a real random variable \(Z\) and any \(\delta>0, p\geq1\) we have that
\begin{equation*}
    \begin{split}
        &\mathds{P}_{Z}\left(\Big|Z-\E_{Z}[Z]\Big|>\delta\right)\leq\min\left\{1,\frac{2^p\E_{Z}[|Z|^p]}{\delta^p}\right\}.
    \end{split}
\end{equation*}
\end{lemma}
\begin{proof}
Indeed, we have
\begin{equation*}
    \begin{split}
        \mathds{P}_{Z}\left(\Big|Z-\E_{Z}[Z]\Big|>\delta\right)
        & \leq \frac{\E_{Z}\bigg[\Big|Z-\E_{Z}[Z]\Big|^p\bigg]}{\delta^p}
        \leq \frac{\E_Z\left[\left(|Z|+|\E_Z[Z]|\right)^p\right]}{\delta^p}\\
        & \leq \frac{2^{p-1}\E_Z\left[|Z|^p+|\E_Z[Z]|^p\right]}{\delta^p}
        \leq \frac{2^p\E_{Z}[|Z|^p]}{\delta^p},
    \end{split}
\end{equation*}
as desired.
\end{proof}

At some point, it will actually be more convenient to work with the following one-sided version of the cut norm (which is a semi-norm):
\[\|\mc{A}\|^+_\pcut = \frac{1}{n^2} \max_{S,T \subseteq[n]} \sum_{i\in S,j\in T} \mc{A}_{ij}\]
for an \(n\times n\) matrix \(\mc{A}\), and
\[\|U\|^+_\pcut = \sup_{S,T \subseteq[0,1]} \int_{S\times T} U(x, y) dx dy\]
for a function $U\in L^1([0,1]^2)$. We note that
\begin{equation}\label{eq:cut_vs-cut+}
\|\mc{A}\|_\pcut= \max\{\|\mc{A}\|^+_\pcut, \|-\mc{A}\|^+_\pcut\}, \quad \text{ and } \quad \|U\|_\pcut= \max\{\|U\|^+_\pcut, \|-U\|^+_\pcut\}. 
\end{equation}

For a stepfunction $W\in L^1_{sym}([0,1]^2)$ over the steps $[(i-1)/n,i/n]$ ($i\in[n]$), let us further define $W_{i,j} \in \mathbb{R}$ as the a.e. value on $[(i-1)/n,i/n]\times[(j-1)/n,j/n]$, and
\[W(Z_1, Z_2) := \sum_{i\in Z_1,\,j\in Z_2} W_{i,j}\]
for finite subsets \(Z_1, Z_2 \subset [n]\).
Similarly, for any \(\mc{A} \in\mathbb{R}^{n\times n} \), any set \(Q_1\) of rows and any set \(Q_2\) of columns, we write \[\mc{A}(Q_1, Q_2) = \sum_{i\in Q_1,j \in Q_2} \mc{A}_{ij}.\]

Note that this means
\begin{equation*}
\|\mc{A}\|_\pcut = \frac{1}{n^2} \max_{S,T \subseteq[n]} \left|\mc{A}(S,T)\right| \quad \text{ and }\quad \|\mc{A}\|^+_\pcut = \frac{1}{n^2} \max_{S,T \subseteq[n]} \mc{A}(S,T).\end{equation*}

Also, simple monotonicity arguments for step functions yield the corresponding result for step functions.
\begin{lemma}\label{le:step}
For any stepfunction $W\in L^1_{sym}([0,1]^2)$ over the steps $[(i-1)/n,i/n]$ ($i\in[n]$), we have
\[
\|W\|_\pcut= \frac{1}{n^2} \max_{S,T \subseteq[n]} \left|W(S,T)\right|.
\]
\end{lemma}
\begin{proof} Consider arbitrary measurable subsets $\mc{S},\mc{T}\subset [0,1]$, and let $i\in[n]$.

If $\int_{[(i-1)/n,i/n]\times T} W\geq 0$, then 
\[
\int_{(\mc{S}\setminus[(i-1)/n,i/n])\times\mc{T}} W\leq\int_{\mc{S}\times\mc{T}}W\leq\int_{(\mc{S}\cup[(i-1)/n,i/n])\times\mc{T}} W,
\]
and if $\int_{[(i-1)/n,i/n]\times T} W< 0$, then
\[
\int_{(\mc{S}\cup[(i-1)/n,i/n])\times\mc{T}} W\leq \int_{\mc{S}\times\mc{T}} W\leq\int_{(\mc{S}\setminus[(i-1)/n,i/n])\times\mc{T}} W.
\]
 This means that fixing the set $\mc{T}$, the absolute value of integral over $\mc{S}\times \mc{T}$ can be maximized by setting each $\lambda(\mc{S}\cap[(i-1)/n,i/n])$ to either $0$ or $1/n$. The same is valid swapping the roles of $\mc{S}$ and $\mc{T}$, and hence the maximum is achieved when each $\lambda(\mc{S}\cap[(i-1)/n,i/n])$ and $\lambda(\mc{T}\cap[(i-1)/n,i/n])$ is either $0$ or $1/n$. In other words, for any $\mc{S}$ and $\mc{T}$, there exist subsets $S,T\subset[k]$ such that
\[
\left|\int_{\mc{S}\times\mc{T}} W\right|\leq \left|\int_{\left(\bigcup_{i\in S}[(i-1)/n,i/n]\right)\times\left(\bigcup_{j\in T}[(j-1)/n,j/n]\right)} W\right|= \frac{1}{n^2}|W(S,T)|,
\]
and we are done.
\end{proof}

For a.e. \(x\in [0,1]\), let us write
\[ U_x:=\E_{z\in [0,1]}\left[ \big|U(x,z)\big|\right] =\int\limits_0^1 \big|U(x,z)\big| dz.\]

\begin{definition}\label{def:L}
Let $\vb,\va>0$ and $k\geq 2$ an integer. Let $X=(X_1, X_2, \dots, X_k) \in [0,1]^k$ be a uniform random $k$-vector and let $U\in L^1_{sym}([0,1]^2)$ be a kernel. We define
\[
\Delta_{U}(X):=\max_{j\in [k]}\left\{\big| U_{X_j}- \|U\|_1\big|, \left|\frac{\sum_{i\in[k]\setminus\{j\}}\left|U(X_i,X_j)\right|}{k-1}-\|U\|_{1}\right|\right\},
\]
and
$ L_{\vb,\va}^0\subseteq [0,1]^k$ as
\begin{equation*}
    \begin{split}
       L_{\vb,\va}^0:= & \left\{X\in [0,1]^k \quad \Big| \quad \Delta_U(X) 
       \leq \vb k^{\va} \|U\|_1 \right\}.
    \end{split}
\end{equation*}
\end{definition}

\bigskip

 We claim that  \(\mathds{P}_{X} \left(X\in L_{\vb,\va}^0\right)\)
is
 very close to \(1\),
i.e., the section integrals and the averages are both highly concentrated around their expectation \(\|U\|_1\).

\begin{prop}\label{Prop:L_Chebyshev}
Let $k\geq 1$ be an integer, $X=(X_1, X_2, \dots, X_k) \in [0,1]^k$ a uniform random $k$-vector. For arbitrary \(p > \max\left\{1,\frac{1}{\va}\right\}\), we have
\begin{equation*}
    \begin{split}
        \mathds{P}_{X} \left(X\in L_{\vb,\va}^0\right) \geq & 1 -2 k\cdot  \frac{2^p\|U\|_p^p }{\vb^p k^{\va p}\|U\|_1^p}   \rightarrow 1  \quad \text{ as } k\rightarrow \infty \\
    \end{split}
\end{equation*}
for any kernel $U\in L^p_{sym}([0,1]^2)$.
\end{prop}

\begin{proof}
For an arbitrary \(j\in [k]\) substituting \(Z=\frac{\sum_{i\in[k]\setminus\{j\}} \left|U(X_i,X_j)\right|}{k-1}\) into the generalised Chebyshev's inequality (Lemma~\ref{Lemma:generalised_Chebyshev}) we obtain that
\begin{equation*}
\begin{split}
    \mathds{P}_{X} \left(\left| \frac{\sum_{i\in[k]\setminus\{j\}} \left|U(X_i,X_j)\right|}{k-1}- \|U\|_1 \right|>\delta_0\right)
    & \leq \min\left\{ 1,\frac{ 2^p \E_{X}\left[\left| \frac{\sum_{i\in[k]\setminus\{j\}} \left|U(X_i,X_j)\right|}{k-1} \right|^p\right]}{\delta_0^p}\right\}.
\end{split}
\end{equation*}
By the convexity of the function \(x^{p}\),
\begin{equation*}
    \begin{split}
        \E_{X}\left[\left| \frac{\sum_{i\in[k]\setminus\{j\}} \left|U(X_i,X_j)\right|}{k-1} \right|^p\right] 
        & \leq \E_{X}\left[\frac{\sum_{i\in[k]\setminus\{j\}} \left|U(X_i,X_j)\right|^p}{k-1}\right] \\
        & = \E_{(x,y)\in [0,1]^2}\left[ \left|U(x,y)\right|^{p}\right] = \|U\|_{p}^{p},
    \end{split}
\end{equation*}
a constant in \(k\).

Hence for all \(p\geq 1\)
\begin{equation*}
\begin{split}
    \mathds{P}_{X} \left(\left|\frac{\sum_{i\in[k]\setminus\{j\}} \left|U(X_i,X_j)\right|}{k-1}- \|U\|_1 \right|>\delta_0\right)
    & \leq \min\left\{ 1,\frac{ 2^p\|U\|_p^p }{\delta_0^p}\right\}.
\end{split}
\end{equation*}

On the other hand, substituting \(Z= U_{x}\) into the generalised Chebyshev's inequality (Lemma~\ref{Lemma:generalised_Chebyshev}) we obtain that
\begin{equation*}
\begin{split}
    \mathds{P}_{x\in [0,1]} \left(\big| U_{x}- \|U\|_1 \big|>\delta_1\right)
    \leq \min\left\{ 1,\frac{ 2^p  \E_{x\in [0,1]}\left[\left| U_{x}\right|^p\right]}{\delta_1^p}\right\}.
\end{split}
\end{equation*}
Again by Jensen's inequality,
\begin{equation*}
    \begin{split}
    \E_{x\in [0,1]}\left[\left(U_{x}\right)^p\right] \leq \E_{(x,y)\in [0,1]^2}\left[\left|U(x,y)\right|^p\right]  = \|U\|_{p}^{p}.
    \end{split}
\end{equation*}

Thus for any \(p \geq 1\)
\begin{equation*}
\begin{split}
    \mathds{P}_{x\in [0,1]} \left(\big| U_{x}- \|U\|_1 \big|>\delta_1\right)
    & \leq \min\left\{ 1,\frac{ 2^p\|U\|_p^p }{\delta_1^p}\right\}.
\end{split}
\end{equation*}

Finally
\begin{align*}
        &\mathds{P}_{X\in [0,1]^k} \left(X\notin L_{\vb,\va}^0\right)\\
        = & \mathds{P} _{X\in [0,1]^k}\left(\exists j\in [k] : \left(\left|U_{X_j}-\|U\|_1\right| > \vb k^\va\|U\|_1\right) \vee \left(\left|\frac{\sum_{i\in[k]\setminus\{j\}} \left|U(X_i,X_j)\right|}{k-1}-\|U\|_1\right| > \vb k^\va\|U\|_1\right)\right) \\
        \leq & k \cdot \mathds{P}_{x\in [0,1]} \left(\big|U_{x}-\|U\|_1\big| >\vb k^\va\|U\|_1\right) + k\cdot \mathds{P}_{X\in [0,1]^k} \left( \left|\frac{\sum_{i\in [k-1]} \left|U(X_i,X_k)\right|}{k-1}-\|U\|_1\right| > \vb k^\va \|U\|_1\right) \\
       \leq & k \cdot \min\left\{1, \frac{2^p\|U\|_p^p }{(\vb k^\va\|U\|_1)^p}\right\}+ k\cdot \min\left\{1, \frac{2^p\|U\|_p^p }{(\vb k^\va\|U\|_1)^p}\right\}  = 2k \cdot \min\left\{1, \frac{2^p\|U\|_p^p }{(\vb k^\va\|U\|_1)^p}\right\}.
\end{align*}
Hence
\begin{equation*}
    \begin{split}
       & \mathds{P}_{X} \left(X\in L_{\vb,\va}^0\right) \geq 1 - 2k \cdot \min\left\{1, \frac{2^p\|U\|_p^p }{(\vb k^\va\|U\|_1)^p}\right\} 
        \geq  1-2k\cdot \frac{2^p\|U\|_p^p }{\vb^p k^{\va p}\|U\|_1^p} \stackrel{k\to\infty}{\longrightarrow} 1,
    \end{split}
\end{equation*}
provided $\va p>1$.
\end{proof}

%%%%%%%%%%%%%%%%%%%
%
%
\section{First Sampling Lemma: Bounding the ``dispersion''}\label{Section:Dispersion}
%
%
%%%%%%%%%%%%%%%%%%%

To simplify the statement of various results and avoid repetition, in what follows, the notation $*\in\left\{\text{ },+\right\}$ means that the superscript $*$ is used to either denote the $+$ symbol, or the empty symbol. For instance this means that when $*=+$, then $\|\cdot\|^*_{\pcut}$ is the one-sided cut norm, whereas when $*$ is the empty symbol, then $\|\cdot\|^*_{\pcut}$ is the ordinary cut norm.

\begin{definition}\label{def:martingale}
Let $k\geq 1$ be an integer, $X=(X_1, X_2, \dots, X_k) \in [0,1]^k$ a uniform random $k$-vector and let $U\in L^1_{sym}([0,1]^2)$ be a kernel.
Let the martingale $(\mc{M}_b(X))_{1\leq b\leq k+1}$
be defined through 
$\mc{M}_b (X):=\E_{X_{[k]\setminus [b-1]}}\|\mc{U}_X\|_\pcut$ for all $b\in[k+1]$ (in particular, $\mc{M}_{k+1}(X)=\|\mc{U}_X\|_\pcut$).
 
Further, define the martingale $(\mc{M}_b^+(X))_{1\leq b\leq k+1}$
as $\mc{M}_b^+(X):=\E_{X_{[k]\setminus [b-1]}}\left[\|\mc{U}_X\|_\pcut^+\right]$ for all $b\in[k+1]$ (in particular, $\mc{M}_{k+1}^+(X)=\|\mc{U}_X\|_\pcut^+$).
\end{definition}

Given an integer $a\in [k]$, let $Y^{(a)}$ be the random $k$-vector obtained by replacing the $a$-th term $X_a$ in $X$ by $X_0$, i.e., $Y^{(a)}=(X_1,\ldots,X_{a-1},X_0,X_{a+1},\ldots, X_k)$.

\begin{lemma}\label{le:martingale_diff}
Let $U\in L^1_{sym}([0,1]^2)$ be a kernel, $k\geq 1$ an integer, $X=(X_1, X_2, \dots, X_k) \in [0,1]^k$ a uniform random $k$-vector and $a\in[k]$. Let further $*\in\left\{\text{ },+\right\}$ (i.e., the below is valid both for the cut norm and its one-sided version, and both martingales). We then have the following.
\begin{enumerate}
\item[($\mc{A}$)] Almost surely
\begin{equation*}
 \left| k^2 \cdot \|\mc{U}_X\|_{\pcut}^*-k^2 \cdot \|\mc{U}_{Y^{(a)}}\|_{\pcut}^* \right| \leq 2 \sum_{j \in[k]\setminus \{a\}} \left|U(X_0,X_j)\right|+ 2 \sum_{ j\in[k]\setminus \{a\}} \left|U(X_a,X_j) \right|.
\end{equation*}
\item[($\mc{B}$)] Almost surely
\begin{equation*}
\left|\mc{M}_{a+1}^*(X)-\mc{M}_{a+1}^*(Y^{(a)})\right| 
\leq  \frac{2}{k^2} \left[\sum_{j \in [a-1]} \left(\big|U(X_0,X_j)\big|+ \big|U(X_a,X_j) \big|\right) + (k-a) \cdot \left(U_{X_0}+ U_{X_a}\right)\right].
\end{equation*}
\item[($\mc{C}$)] Almost surely
\begin{equation*}
    \left|\mc{M}_a^*(X)-\mc{M}_{a+1}^*(X)\right|\leq \frac{2}{k^2} \left[\sum_{j \in [a-1]} \big|U(X_a,X_j)\big|+ \sum_{j \in [a-1]} U_{X_j} + (k-a) \cdot U_{X_a}+ (k-a) \cdot \|U\|_1\right].
\end{equation*}
\end{enumerate}
\end{lemma}

\begin{proof} 
Apart from a single step in the proof of part $(\mc{A})$, the proofs for $\mc{M}$ and the cut norm are identical to the proofs for $\mc{M}^+$ and the corresponding $\|\cdot\|^+$.

\underline{Part $(\mc{A})$:}
For any \(S,T\subseteq [k]\) the function \(\mc{U}_X\left(S,T\right)\) is a function of the independent random variables \(X_\ell, \ell \in [k]\), and if we change the value of one of these \(X_a\) to \(X_0\), the sum \(\mc{U}_X(S,T)\) changes to \(\mc{U}_{Y^{(a)}}(S,T)\). Note that
\begin{equation*}
    \begin{split}
       & \left|\left|\mc{U}_{X}(S, T)\right|-\left|\mc{U}_{Y^{(a)}}(S,T)\right| \Big.\right|
        \leq  \left| \mc{U}_{X}(S, T) - \mc{U}_{Y^{(a)}}(S,T) \Big.\right| 
        = \left|\sum_{i\in S, j\in T} \left( U(X_i,X_j)-U(Y^{(a)}_i, Y^{(a)}_j)\right)\right|\\
     \leq & \sum_{i\in S, j\in T} \left|\Big.U(X_i,X_j)-U(Y^{(a)}_i, Y^{(a)}_j)\right|
        \leq  2 \sum_{j\in[k]\setminus \{a\}} \big| U(X_a,X_j) \big| + 2\sum_{j\in[k]\setminus \{a\}} \big|U(X_0,X_j)\big|.
    \end{split}
\end{equation*}
For the last inequality, note that the terms with $i,j\neq a$ cancel out, and if for example \(a \in S\) and  \(a \notin T\) then the sum changes by at most
    \begin{equation*}
        \begin{split}
            \sum_{ j\in T} \left|\Big.U(X_0,X_j)-U(X_a,X_j) \right| & \leq \sum_{ j\in T} \left(\left|U(X_0,X_j)\right|+ \left|U(X_a,X_j)\right|\right) \\
            & \leq \sum_{ j\in[k]\setminus \{a\}} \left|U(X_0,X_j)\right|+ \sum_{j\in[k]\setminus \{a\}} \left|U(X_a,X_j) \right|,
        \end{split}
    \end{equation*}
    with the other cases for $a$ working in a similar fashion.

Hence, recalling Lemma \ref{le:step}, we have
\begin{equation*}
    \begin{split}
       & \left| k^2 \cdot \|\mc{U}_X\|_{\pcut}-k^2 \cdot \|\mc{U}_{Y^{(a)}}\|_{\pcut} \right|= \left|\max_{S,T \subseteq[k]} \left|\mc{U}_X(S,T)\right| - \max_{S,T \subseteq[k]} \left|\mc{U}_{Y^{(a)}}(S,T)\right|\right| \\
       & \quad \quad \quad \leq 2 \sum_{ j\in[k]\setminus \{a\}} \left|U(X_0,X_j)\right|+ 2 \sum_{ j\in[k]\setminus \{a\}} \left|U(X_a,X_j) \right|,
    \end{split}
\end{equation*}
but also
\begin{equation*}
    \begin{split}
       & \left| k^2 \cdot \|\mc{U}_X\|_{\pcut}^+-k^2 \cdot \|\mc{U}_{Y^{(a)}}\|^+_{\pcut} \right|= \left|\max_{S,T \subseteq[k]} \mc{U}_X(S,T) - \max_{S,T \subseteq[k]} \mc{U}_{Y^{(a)}}(S,T)\right| \\
       & \quad \quad \quad \leq 2 \sum_{ j\in[k]\setminus \{a\}} \left|U(X_0,X_j)\right|+ 2 \sum_{ j\in[k]\setminus \{a\}} \left|U(X_a,X_j) \right|,
    \end{split}
\end{equation*}
proving inequality $(\mc{A})$.

\underline{Part $(\mc{B})$:}
For the second inequality, we are fixing \(X_1,X_2, \dots X_{a}\), and taking the expected value over $X_{a+1},\ldots,X_k$.

\begin{equation*}
    \begin{split}
        & \left|\mc{M}_{a+1}^*(X)-\mc{M}_{a+1}^*(Y^{(a)})\right|=  \left|\E_{X_{[k] \setminus [a]}
        }
        \left[ \|\mc{U}_X\|_\pcut^* - \|\mc{U}_{Y^{(a)}}\|_\pcut^*\right]\right|  \leq \E_{X_{[k] \setminus [a]}
        }\left| \|\mc{U}_X\|_\pcut^* - \|\mc{U}_{Y^{(a)}}\|_\pcut^*\right| \\
        & \stackrel{(\mc{A})}{\leq} \frac{2}{k^2} \E_{X_{[k] \setminus [a]}
        } \left[\sum_{j\in[k]\setminus \{a\}} \left(\big|U(X_0,X_j)\big|+ \big|U(X_a,X_j) \big|\right)\right]\\
        &  = \frac{2}{k^2} \left[\sum_{j\in [a-1]} \left(\big|U(X_0,X_j)\big|+ \big|U(X_a,X_j) \big|\right) +\E_{X_{[k] \setminus [a]}
        } \left[\sum_{ j \in [k] \setminus [a]} \left(\big|U(X_0,X_j)\big|+ \big|U(X_a,X_j) \big|\right)\right] \right] \\
        & = \frac{2}{k^2} \left[\sum_{j\in [a-1]} \left(\big|U(X_0,X_j)\big|+ \big|U(X_a,X_j) \big|\right) + \left|\{j\in[k] \setminus [a]\}\right| \cdot \left(U_{X_0}+ U_{X_a}\right)\bigg.\right] \\
        & = \frac{2}{k^2} \left[\sum_{j\in [a-1]} \left(\big|U(X_0,X_j)\big|+ \big|U(X_a,X_j) \big|\right) + (k-a) \cdot \left(U_{X_0}+ U_{X_a}\right)\right].
    \end{split}
\end{equation*}

\underline{Part $(\mc{C})$:}
Finally, we look at the martingale difference:
\begin{equation*}
    \begin{split}
    &\left|\mc{M}_a^*(X)-\mc{M}_{a+1}^*(X)\right|=
    \left|\Big.\E_{X_{[k] \setminus [a-1]}         
    } \|\mc{U}_X\|_\pcut^* - \E_{X_{[k] \setminus [a]}
    } \|\mc{U}_{X}\|_\pcut^* \right|
    \\
    & =\left|\Big.\E_{X_0,X_{[k] \setminus [a]}
    } \|\mc{U}_{Y^{(a)}}\|_\pcut^* - \E_{X_{[k] \setminus [a]}
    }\|\mc{U}_{X}\|_\pcut^* \right|
    \\
    & =\left| \E_{X_0} \left[\Big.\E_{X_{[k] \setminus [a]}
    }\left[ \|\mc{U}_{Y^{(a)}}\|_\pcut^* - \|\mc{U}_{X}\|_\pcut^* \right]\right]\right| 
    \leq \E_{X_0} \left|\Big.\E_{X_{[k] \setminus [a]}
    }\left[ \|\mc{U}_{Y^{(a)}}\|_\pcut^* - \|\mc{U}_{X}\|_\pcut^*\right]\right|
    \\
    & \stackrel{(\mc{B})}{\leq} \frac{2}{k^2} \E_{X_0}\left[\sum_{j\in [a-1]} \left(\big|U(X_0,X_j)\big|+ \big|U(X_a,X_j) \big|\right) + (k-a) \cdot \left(U_{X_0}+ U_{X_a}\right)\right]
    \\
    & = \frac{2}{k^2} \left[\sum_{j\in [a-1]} \big|U(X_a,X_j)\big|+ \E_{X_0}\left[\sum_{j\in [a-1]}\big|U(X_0,X_j) \big|\right] + (k-a) \cdot U_{X_a}+ (k-a) \cdot \E_{X_0}\left[U_{X_0}\right]\right]
    \\
    & =\frac{2}{k^2} \left[\sum_{j\in [a-1]} \big|U(X_a,X_j)\big|+ \sum_{j\in [a-1]} U_{X_j} + (k-a) \cdot U_{X_a}+ (k-a) \cdot \|U\|_1\right].
    \end{split}
\end{equation*}
\end{proof}

This now allows us to bound the size of the sets where the martingale difference is large, and apply Lemma \ref{Lemma:generalised_azuma}.

\begin{lemma}\label{le:Azuma_lower}
Let $U\in L_{sym}^1([0,1]^2)$ be a kernel and $k\geq 1$ an integer. Let further $*\in\left\{\text{ },+\right\}$, and
\[
H_{\alpha,a}^*:= \left\{X\in [0,1]^k \quad \bigg| \quad \left|\mc{M}^*_a(X)-\mc{M}^*_{a+1}(X)\right| \geq \alpha\right\}
\]
for $a\in [k]$ and $\alpha>0$.
Setting $\alpha_0:=\frac{6}{k}\|U\|_1\left(1+\vb k^\va\right)$ for some $\vb,\va>0$, we then have 
\[
\bigcup_{a\in[k]} H_{\alpha_0,a}^* \subseteq (L^0_{\vb,\va})^c,
\]
where $(L^0_{\vb,\va})^c$ denotes the complement of $L^0_{\vb,\va}$.
Consequently, for any $p>\frac1\va$, $U\in L_{sym}^p([0,1]^2)$, $\lambda>0$, integer $k\geq 1$, and $X=(X_1, X_2, \dots, X_k) \in [0,1]^k$ uniform random $k$-vector, we have
\begin{equation*}
        \mathds{P}_{X} \left(\left|\Big.\|\mc{U}_X\|_\pcut^* - \E_{X}\left[ \|\mc{U}_X\|_\pcut^*\right]\right| \geq 2 \lambda \alpha_0\sqrt{k}\right)
        \leq 2e^{-2 \lambda^2} + 2k \min\left\{1, \frac{2^p\|U\|_p^p}{\vb^p k^{\va p}\|U\|_1^p}\right\}.
\end{equation*}
\end{lemma}

\begin{proof}
By Lemma \ref{le:martingale_diff} Part $(\mc{C})$, we have for any $X\in L_{\vb,\va}^0$ that
\begin{equation*}
    \begin{split}
       & \left|\mc{M}^*_a(X)-\mc{M}^*_{a+1}(X)\right|\leq \frac{2}{k^2} \left[\sum_{j \in [a-1]} \big|U(X_a,X_j)\big|+ \sum_{j \in [a-1]} U_{X_j} + (k-a) \cdot U_{X_a}+ (k-a) \cdot \|U\|_1\right]\\
       \leq & \frac{2}{k^2} \left[\sum_{j\in[k]\setminus\{a\}} \big|U(X_a,X_j)\big|+ \sum_{j \in [a-1]} U_{X_j} + (k-a) \cdot U_{X_a}+ (k-a) \cdot \|U\|_1\right]\\
         \leq &\frac{2}{k^2} \left[(k-1) \|U\|_1 \left(1+ \vb k^\va\right)+ (a-1) \|U\|_1 \left(1+ \vb k^\va\right) + (k-a) \cdot \|U\|_1\left(1 + \vb k^\va\right) + (k-a) \cdot \|U\|_1\big.\right]\\
      =& \frac{2}{k^2} \left[\big. (3k-a-2) \|U\|_1 + (2k-2) \vb k^\va \|U\|_1\right] < \frac{6k}{k^2} \|U\|_1\left(1 + \vb k^\va\right) =\alpha_0.
    \end{split}
\end{equation*}

Note that \(\alpha_0\) depends only on \(k, \vb, \va\) (and not on \(a\)).
By applying Lemma \ref{Lemma:generalised_azuma} to the function $f(X):=\|\mc{U}_X\|^*_\pcut: [0,1]^k\to\mathbb{R}$ and using Proposition \ref{Prop:L_Chebyshev}, we then obtain

\begin{equation*}
    \begin{split}
  \mathds{P}_{X} \left(\left|\Big.\|\mc{U}_X\|^*_\pcut - \E_X\left[ \|\mc{U}_X\|^*_\pcut\right]\right|  \geq 2 \lambda \sqrt{k\cdot  \alpha_0^2}\right)
       & \leq 2 e^{-2 \lambda^2} + \mathds{P}_{X} \left(X\in \cup_{a\in[k]}H^*_{\alpha_0,a}\right) \\
       & \leq 2 e^{-2 \lambda^2} + \mathds{P}_{X} \left(X\in (L_{\vb,\va}^0)^c\right) \\
       & \leq 2 e^{-2 \lambda^2} + 2k \cdot \min\left\{1, \frac{2^p\|U\|_p^p }{\vb^p k^{\va p}\|U\|_1^p}\right\}.
    \end{split}
\end{equation*}
\end{proof}

%%%%%%%%%%%%%%%%%%%
%
%
\section{Lower bound on the ``systematic error'' term}\label{sect:lower}
%
%
%%%%%%%%%%%%%%%%%%%

Bounding $\E[\|\mc{U}_X\|_\pcut] - \|U\|_\pcut$ from below is rather straight forward, and both the bound and its proof carry over from the bounded kernel case.

\begin{lemma}\label{le:lower_expect}
For any kernel $U \in L^1_{sym}([0,1]^2)$, $k\geq 1$ integer and $X=(X_1, X_2, \dots, X_k) \in [0,1]^k$ uniform random $k$-vector, we have
\begin{equation*}
    \E_X\left[ \|\mc{U}_X\|_\pcut \right] -\|U\|_\pcut \geq  - \frac{1}{k}\|U\|_\pcut.
\end{equation*}
\end{lemma}
\begin{proof} Let \(S_1, S_2 \subset [0, 1]\) be two measurable sets, and define the induced random finite sets 
\[
\mc{S}_i:=\left\{j\in[k]: X_j\in S_i\right\} \quad i=1,2.
\]
By Lemma \ref{le:step} we then have
\[\|\mc{U}_X\|_\pcut \geq \frac{1}{k^2} \left|\mc{U}_X(\mc{S}_1, \mc{S}_2)\right|.\]
Fixing \(S_1,S_2\) and taking the expectation over the k-vector \(X\), we obtain
\begin{equation*}
    \begin{split}
        \E_X\left[ \|\mc{U}_X\|_\pcut\right] \geq & \frac{1}{k^2} \E_X \left|\Big. \mc{U}_X(\mc{S}_1, \mc{S}_2) \right| \geq \frac{1}{k^2} \left|\Big.\E_X\left[\mc{U}_X(\mc{S}_1, \mc{S}_2)\right] \right| \\
        = & \frac{k - 1}{k} \left| \int_{S_1 \times S_2} U(x, y) dx dy \right|,
    \end{split}
\end{equation*}
as $\mc{U}[X]$ is 0 on the main diagonal by definition.

Taking the supremum over all measurable sets \(S_1, S_2\), this leads to
\begin{equation*}
    \E_X\left[ \|\mc{U}_X\|_\pcut \right] \geq \frac{k-1}{k}\|U\|_\pcut,
\end{equation*}
and hence
\begin{equation}\label{eq:lower2}
    \E_X\left[ \|\mc{U}_X\|_\pcut \right] -\|U\|_\pcut \geq  - \frac{1}{k}\|U\|_\pcut.
\end{equation}
\end{proof}

\begin{remark}
Note that the lower bound in \eqref{eq:lower2} is of an order that will turn out to be negligible compared to the bound on the dispersion term, and thus may seem of little importance. However,  the true inequality here is actually the equivalent
\begin{equation}\label{eq:lower1}
    \E_X\left[ \frac{k}{k-1} \|\mc{U}_X\|_\pcut \right] \geq\|U\|_\pcut,
\end{equation}
which corresponds to a different normalisation of the cut norm for the sample $\mc{U}_X$. Indeed, just as for simple graphs, samples from graphons result in $k\times k$ matrices that have diagonal entries all $0$, and so one may argue that the ``correct'' normalisation for the cut norm would be a factor $1/k(k-1)$ instead of $1/k^2$: inequality \eqref{eq:lower1} is satisfied with equality for the constant $1$ graphon (corresponding to the complete graph case).
\end{remark}

%%%%%%%%%%%%%%%%%%%
%
%
\section{Upper bound on the ``systematic error'' term}\label{Section:upper_bound}
%
%
%%%%%%%%%%%%%%%%%%%
We shall present two approaches to bounding the "systematic error" $\E_X\left[ \|\mc{U}_X\|_\pcut \right] -\|U\|_\pcut$ from above, which in the bounded kernel case has been shown to be less than or equal to $\frac{8\|U\|_\infty}{k^{1/4}}$ (\cite[Inequality (4.15)]{Borgs}). Our main approach -- and somewhat surprisingly the one yielding better bounds, works by truncating $U$ at an appropriate, $k$ dependent level, and then bounding the various error terms that arise, using an intermediate step of the original bounded kernel First Sampling Lemma as a black box along the way. The second approach adapts the full original proof to the unbounded setting, and yields a slightly weaker bound of order $k^{-\frac{1}{4}+\frac{1}{2p}}(\ln k)^{1/4}$, as opposed to $k^{-\frac{1}{4}+\frac{1}{4p}}$ for our main approach. We chose to include the second approach for two reasons. On the one hand, there are intermediate results that in themselves are interesting, and fall outside of the path of the main approach. On the other hand, comparing these approaches leads to an interesting open question. Namely, the approach leading to a weaker bound being the one that closely follows the original proof may be an indication that this original approach is not sharp even in the bounded kernel case, and that there may be room for improvements, which would in turn impact the unbounded kernel case as well.

As the second approach yields a slightly weaker bound, we shall only provide the interesting intermediate results and a skeleton proof here, with details relegated to the \nameref{sect:Appendix}.

\subsection{Main approach: truncation}\label{sect:truncation}

Let us follow a direct truncation approach to the problem. Namely, we truncate the unbounded kernel $U\in L^1_{sym}([0,1])$ at a certain threshold, turning it into a bounded kernel to which we can apply intermediate results from the original First Sampling Lemma. To finish, we then bound the various error terms arising from the truncation.

For each $k$, let $U^*_k$ denote a truncated version of $U$.
For each $k\in\mathbb{N}$ we have a threshold $f(k)>0$ with $\lim_{k\to\infty} f(k)=\infty$ such that $U_k^*(x,y)=U(x,y)$ as long as $|U(x,y)|\leq f(k)$, otherwise $|U_k^*(x,y)|=f(k)$ and $U_{k}^*$ retains the same sign as $U$ at every point.

Then we have the following.
\begin{prop}\label{prop:upper_expect_truncation}
    Let $k\geq 2$ be an integer and $p\geq 2$ . Then for any kernel $U\in L^p_{sym}([0,1]^2)$, we have
\begin{align*}
\E_X(\|\mc{U}_X\|_\pcut)-\|U\|_\pcut\leq 30\|U\|_pk^{-\frac14+\frac{1}{4p}}.
\end{align*}
\end{prop}

\begin{proof}
Let $\mc{U}_{k,X}^*$ denote the random $k$-sample obtained from $U_k^*$.
Then we have the following.
\begin{align*}
\E_X(\|\mc{U}_X\|_\pcut)-\|U\|_\pcut=&(\E_X(\|\mc{U}_X\|_\pcut)-\E_X(\|\mc{U}^*_{k,X}\|_\pcut))+(\E_X(\|\mc{U}^*_{k,X}\|_\pcut)-\|U_k^*\|_\pcut)+(\|U_k^*\|_\pcut-\|U\|_\pcut)\\
\leq&\E_X(\|\mc{U}_X\|_\pcut-\|\mc{U}^*_{k,X}\|_\pcut)+(\E_X(\|\mc{U}^*_{k,X}\|_\pcut-\|U_k^*\|_\pcut)+(\|U_k^*\|_\pcut-\|U\|_\pcut)
\end{align*}
The second term on the RHS is covered by the bounded kernel case, though the error bound has to be scaled with the $L^\infty$ norm of $U^*_k$ which is at most $f(k)$. By \cite[Inequality (4.15)]{Borgs} we then obtain
\[
\E_X(\|\mc{U}^*_{k,X}\|_\pcut-\|U_k^*\|_\pcut)\leq \frac{8f(k)}{k^{1/4}}.
\]
On the other hand, the triangle inequality yields
\begin{align*}
    \|\mc{U}_X\|_\pcut-\|\mc{U}^*_{k,X}\|_\pcut\leq \|\mc{U}_X-\mc{U}_{k,X}^*\|_\pcut\leq \|\mc{U}_X-\mc{U}_{k,X}^*\|_{_{1}}
\end{align*}
and
\begin{align*}
    \|U_k^*\|_\pcut-\|U\|_\pcut\leq\|U_k^*-U\|_\pcut\leq\|U_k^*-U\|_1,
\end{align*}
hence the first and third terms can be bounded as follows.
\begin{align*}
    \E_X(\|\mc{U}_X\|_\pcut-\|\mc{U}^*_{k,X}\|_\pcut)+(\|U_k^*\|_\pcut-\|U\|_\pcut)
    & \leq \E_X(\|\mc{U}_X-\mc{U}_{k,X}^*\|_{_{1}})+\|U_k^*-U\|_1 \\
    & =\left(\frac{k(k-1)}{k^2}+1\right)\|U_k^*-U\|_1.
\end{align*}

Hence it only remains to bound the effect of the truncation, which can be done using the $L^p$ norm. Indeed
\begin{align*}
\|U_k^*-U\|_1 & = \int_{|U|>f(k)} (|U|- f(k)) = \int_{0}^{\infty} \mathds{P} \left(|U|\geq t + f(k)\right)\,dt \\
& =\int_{f(k)}^\infty \mathds{P} \left(|U|\geq t'\right)\,dt'=\int_1^\infty f(k)\mathds{P} \left(|U|\geq sf(k)\right) \,ds.
\end{align*}
Applying Lemma \ref{Lemma:generalised_Chebyshev} we obtain for $s\geq1$
\begin{align*}
\mathds{P} \left(|U|\geq sf(k)\right)=& \mathds{P} \left(|U|-\|U\|_1\geq sf(k)-\|U\|_1\right)\leq
\mathds{P} \left(|U|-\|U\|_1\geq s(f(k)-\|U\|_1)\right)\\
\leq& \frac{2^p\|U\|_p^p}{s^p(f(k)-\|U\|_1)^p},
\end{align*}
leading to
\begin{align*}
\|U_k^*-U\|_1\leq& \int_1^\infty f(k)\mathds{P} \left(|U|\geq sf(k)\right) \,ds
\leq \int_1^\infty f(k)\frac{2^p\|U\|_p^p}{s^p(f(k)-\|U\|_1)^p}\,ds\\
\leq& 2^p\|U\|_p^p\frac{f(k)}{(f(k)-\|U\|_1)^p}.
\end{align*}
Hence
\[
\E_X(\|\mc{U}_X\|_\pcut)-\|U\|_\pcut\leq \frac{8f(k)}{k^{1/4}}+2^{p+1}\|U\|_p^p\frac{f(k)}{(f(k)-\|U\|_1)^p},
\]
and setting $f(k):=c k^{\va}$, the first term increases with $\va$, whereas the second decreases. The optimum balance is reached when $\va p=1/4$, and choosing $c=3\|U\|_p\geq3\|U\|_1$ finishes the proof, since $f(k)-\|U\|_1\geq 2\|U\|_pk^\va$.
\end{proof}

\subsection{``Original'' approach: Random $q$-subsamples}\label{subsect:original}
The main issue with proving a high probability upper bound is that the cut norm of a sample is a maximum over all possible subsets of rows and columns, and bounding from above the expectation of a maximum is significantly harder than a maximum of expectations. Even though the corresponding sums are not independent, and some are automatically too small due to the size of the subsets, we a priori still have to contend with $\Theta(4^k)$ choices, each with their own ``bad set'' where the corresponding sum is too far away from $\|U\|_\pcut$, and we cannot obtain good bounds via such a direct approach.

The upper bound on the ``systematic error'' $\E_X\left[ \|\mc{U}_X\|_\pcut \right] -\|U\|_\pcut$ will be achieved by first considering the one-sided cut norm $\|\cdot\|_\pcut^+$, and then introducing an extra layer of randomisation. Namely, we shall make use of a random subset $Q_1$ of the rows and a random subset $Q_2$ of the columns of $\mc{U}_X$ to bound the expectation of $\|\mc{U}_X\|_\pcut^+$, and then apply the generalised Azuma's inequality to bound the deviation from this expected value. The essence of this approach is that it allows us to strike a balance between the bound on the deviation from the expected value, and the size of the corresponding exceptional set.

Let $k\geq 1$ be an integer and $R_1,R_2\subseteq [k]$. We shall view $R_1$ as a set of indices of columns and $R_2$ as a set of indices of rows. Given $U\in L^1_{sym}([0,1]^2)$ and $Z\in[0,1]^k$, we denote by \(R^+_{1,Z}\) the set of columns \(j \in [k]\) for which \(\mc{U}_Z(R_1, \{j\}) > 0\), and by \(R^-_{1,Z}\) the set of columns \(j \in [k]\) for which \(\mc{U}_Z(R_1, \{j\}) < 0\). The sets of rows
\(R^+_{2,Z}, R^-_{2,Z}\) are defined analogously. Note that \(\mc{U}_Z(R_1, R^+_{1,Z}), \mc{U}_Z(R^+_{2,Z}, R_2) \geq 0\) by this definition. Since $k\|\mc{U}_X\|_2=\|((U(X_i,X_j)))_{i\neq j\in[k]}\|_F$ where $\|\cdot\|_F$ denotes the Frobenius norm of a matrix, we have the following.

\begin{lemma}(\cite[Lemma 3]{Alon})\label{Lemma:BS12} Let $k\geq1$ be an integer, $q\in[k]$, and \(R_1, R_2 \subseteq[k]\). Let further \(Q\) be a random \(q\)-subset of \([k]\) and $U\in L^2_{sym}([0,1]^2)$ a kernel. Then for a.e. \(X\in [0,1]^k\),
\[\mc{U}_X(R_1, R_2) \leq \E_{Q}\bigg[\mc{U}_X\left((Q \cap R_2)_X^+, R_2\right)\bigg] + \frac{k^2}{\sqrt{q}}\|\mc{U}_X\|_2.\]
\end{lemma}

The following lemma gives an upper bound on the one-sided cut norm, using the sampling procedure from Lemma~\ref{Lemma:BS12} (i.e. by maximizing only over certain rectangles, at the cost of averaging these estimates). The main point for our purposes will be that (for a fixed \(Q_1\) and \(Q_2\)), the number of index subset pairs to consider is only \(4^q\), as opposed to \(4^k\) in the definition of the cut norm.

\begin{lemma}\label{Lemma:Bplus_upper} Let $U\in L^2_{sym}([0,1]^2)$ be a kernel, $k\geq1$ an integer, $q\in[k]$, and \(Q_1\), \(Q_2\) two independent uniform random \(q\)-subsets of $[k]$. Then for a.e. \(X\in [0,1]^k\)
\begin{equation*}
    \begin{split}
        \|\mc{U}_X\|^+_\pcut \leq & \frac{1}{k^2} \E_{Q_1,Q_2} \left[\max_{\substack{R_1\subseteq Q_1, \\ R_2\subseteq Q_2}} \mc{U}_X(R_{2,X}^+,R_{1,X}^+)\right]+\frac{2}{\sqrt{q}} \|\mc{U}_X\|_2.
    \end{split}
\end{equation*}
\end{lemma}

The proof follows from an iterated application of Lemma~\ref{Lemma:BS12}.\\
Now note that by definition, $-\mc{U}_X(R_{2,X}^-,R_{1,X}^-)=(-\mc{U})_X(R_{2,X}^+,R_{1,X}^+)$, since positive sums for $-U$ correspond exactly to negative sums for $U$. Hence applying the previous lemma to $-U$, we obtain that also for a.e. \(X\in [0,1]^k\)
\begin{equation*}
    \begin{split}
        \|-\mc{U}_X\|^+_\pcut \leq & \frac{1}{k^2} \E_{Q_1,Q_2} \left[\max_{\substack{R_1\subseteq Q_1, \\ R_2\subseteq Q_2}} -\mc{U}_X(R_{2,X}^-,R_{1,X}^-)\right]+\frac{2}{\sqrt{q}} \|\mc{U}_X\|_2.
    \end{split}
\end{equation*}

If we in addition take into consideration that the expectation of a maximum is at least the maximum of the corresponding expectations, we obtain the following.

\begin{cor}\label{cor:EB}
 Let $U\in L^2_{sym}([0,1]^2)$ be a kernel, $k\geq1$ an integer, $q\in[k]$, and \(Q_1\), \(Q_2\) two independent uniform random \(q\)-subsets of $[k]$. Let further $X=(X_1, X_2, \dots, X_k) \in [0,1]^k$ be a uniform random $k$-vector. Then
\begin{equation*}
    \begin{split}
        \|\mc{U}_X\|_\pcut \leq & \frac{1}{k^2} \E_{Q_1,Q_2\subseteq [k]} \left[\max\left\{\max_{\substack{R_1\subseteq Q_1, \\ R_2\subseteq Q_2}} \mc{U}_X(R_{2,X}^+,R_{1,X}^+),\max_{\substack{R_1\subseteq Q_1, \\ R_2\subseteq Q_2}}- \mc{U}_X(R_{2,X}^-,R_{1,X}^-)\right\}\right]+\frac{2}{\sqrt{q}} \|\mc{U}_X\|_2.
    \end{split}
\end{equation*}
\end{cor}

Taking the expectation with respect to $X$ on both sides in the above inequality, we reduce the problem to bounding the expectation of the RHS from above by $\|U\|_\pcut$ plus the error term.
We shall treat the expectations of the two summands on the RHS of the inequality in Corollary~\ref{cor:EB}
separately, with the first term being much trickier to handle.

First we deal with the expectation of the $L^2$ norm. Not too surprisingly, we are able to provide an upper bound in terms of the $L^2$ norm of the original kernel $U$.

\begin{lemma}\label{Lemma:Frobenius} Let $U\in L^2_{sym}([0,1]^2)$ be a kernel, $k\geq 1$ be an integer and $X=(X_1, X_2, \dots, X_k) \in [0,1]^k$  a uniform random $k$-vector. Then
\[
\E_X\left[\|\mc{U}_X\|_2\right]\leq \sqrt{\frac{k-1}{k}} \|U\|_2 \leq \|U\|_2.
\]
\end{lemma}
\begin{proof}
Note that \[\E_X\left[\|\mc{U}_X\|_2^2\right]=\frac1{k^2}\E_X\left[\sum_{i\neq j\in[k]} |U(X_i,X_j)|^2\right]= \frac{k(k-1)}{k^2} \cdot \int_0^1 \int_0^1 U(x,y)^2 dx dy = \frac{k-1}k \|U\|_2^2,\]
since we set $\mc{U}_X$ to 0 on the steps along the main diagonal.
Also, by convexity $\left(\E_X\left[\|\mc{U}_X\|_2\right]\right)^2\leq\E_X\left[\|\mc{U}_X\|_2^2\right]$.
\end{proof}

Next we look at obtaining an upper bound on the expectation of the maxima, 
\[
\E_X\left[\E_{Q_1,Q_2\subseteq [k]}\left[ \max\left\{\max_{\substack{R_1\subseteq Q_1, \\ R_2\subseteq Q_2}} \mc{U}_X(R_{2,X}^+, R_{1,X}^+),\max_{\substack{R_1\subseteq Q_1, \\ R_2\subseteq Q_2}} -\mc{U}_X(R_{2,X}^-,R_{1,X}^-)\right\}\right]\right].
\]

The main difficulty for this term is that whilst we can relatively easily bound each individual expectation $\E_X\left[\E_{Q_1,Q_2\subseteq [k]}\left[\pm \mc{U}_X(R_{2,X}^\pm, R_{1,X}^\pm)\right]\right]$, being able to derive an upper bound on the maximum's expectation requires us to be able to also control what exactly happens on the ``bad sets'' for all choices of $R_1,R_2$ simultaneously. This is in stark contrast to the bounded case, where a trivial infinity norm bound could be used on the maximum as long as the bad sets were small enough to result in a negligible contribution to the expectation.

Fixing the sets $Q_1,Q_2$, we shall first prove an upper bound on \(\mc{U}_X\left(R_{2,X}^+,R_{1,X}^+\right)\) with high probability for a fixed choice of $R_1,R_2$ (using Azuma's inequality), then proceed to treating the uniform problem of a high probability upper bound on \(\max\limits_{R_i \subseteq Q_i} \mc{U}_X\left(R_{2,X}^+,R_{1,X}^+\right) \). By symmetry (plugging in $-U$ for $U$), we obtain a corresponding high probability upper bound on \(\max\limits_{R_i \subseteq Q_i} -\mc{U}_X\left(R_{2,X}^-,R_{1,X}^-\right) \) as well. Integrating then leads to a deterministic upper bound on the expectation.

\noindent Given the sets $Q_1,Q_2$, let us fix the values \(X_i\) for \(i \in Q := Q_1 \cup Q_2\) and take the expected value of \(\mc{U}_X\left(R_{2,X}^+,R_{1,X}^+\right)\) over the rest.
Our approach to an upper bound will be very similar to the one taken to bound the dispersion term.

\begin{definition}
Let $U\in L^1_{sym}([0,1]^2)$ be a kernel. Let $k\geq 1$ be an integer, $Q_1,Q_2\subseteq [k]$ with $|Q_1|=|Q_2|$, and $X=(X_1, X_2, \dots, X_k) \in [0,1]^k$ a uniform random $k$-vector. Define the martingale $(\mc{N}_{(R_1,R_2),\,b}(X))_{b\in[k+1]}$
as 
\[
\mc{N}_{(R_1,R_2),\,b}(X):=\E_{X_{[k]\setminus([b-1]\cup Q)}}\left[\mc{U}_X(R_{2,X}^+,R_{1,X}^+)\right]
\]
for arbitrary \(R_1 \subseteq Q_1, R_2 \subseteq Q_2\) and all $b\in [k+1]$.

In particular, $\mc{N}_{(R_1,R_2),\,1}(X)=\E_{X_{[k]\setminus Q}}\left[\mc{U}_X(R_{2,X}^+,R_{1,X}^+)\right]$ and $\mc{N}_{(R_1,R_2),\,k+1}(X)=\mc{U}_X(R_{2,X}^+,R_{1,X}^+)$.
\end{definition}

Note that \(|Q| \leq 2\vc\). Similarly to Lemma \ref{le:martingale_diff}, we can again bound the martingale differences. However, as we now have one martingale for each choice of $R_1$ and $R_2$, this time around we also need to keep track of the exceptional sets to guarantee that their total measure is still small, so that we can apply the generalised Azuma's inequality Lemma \ref{Lemma:generalised_azuma}, see Lemma \ref{le:Azuma_upper}.

In particular, we obtain the following concentration inequality.

\begin{cor}\label{cor:Azuma_upper}
Let $k\geq 2$, $q\in[k]$ be integers, $p\geq 2$, $\vb\geq1$, $\va>0$, $Q_1,Q_2\subseteq [k]$ with $|Q_1|=|Q_2|=q$. Then for any kernel $U\in L^p_{sym}([0,1]^2)$, there exists a set $S_{k,\vb,\va,Q_1,Q_2}$ of measure at most $\frac{4}{\vb^p k^{\va p}}$ such that for any $X\in L_{\vb,\va}^0\setminus S_{k,\vb,\va,Q_1,Q_2}$ and $*\in\left\{+,-\right\}$ we have
\begin{align*}
\max_{\substack{R_1\subseteq Q_1, \\ R_2\subseteq Q_2}}\left|\mc{U}_X(R_{2,X}^*,R_{1,X}^*) - \E_{X_{[k]\setminus Q}}\left[\mc{U}_X(R_{2,X}^*,R_{1,X}^*)\right]\right|
\leq 4(k-1)\sqrt{k\ln k} \|U\|_1\left(1 + \vb k^\va\right)\sqrt{ \frac{q\vb p(\va+3)}{2}}.
\end{align*}
\end{cor}

This allows us to bound the tail of the distribution of the random variable
\[
\max\left\{\max_{\substack{R_1\subseteq Q_1, \\ R_2\subseteq Q_2}} \mc{U}_X(R_{2,X}^+, R_{1,X}^+),\max_{\substack{R_1\subseteq Q_1, \\ R_2\subseteq Q_2}} -\mc{U}_X(R_{2,X}^-,R_{1,X}^-)\right\},
\]
(Lemma \ref{le:max_upper}), and thereby its expectation.

\begin{prop}\label{prop:max_expect}
Let $k\geq 2$, $q\in[k]$ be integers, $p\geq 2$, and $\va>0$. Then for any kernel $U\in L^p_{sym}([0,1]^2)$, we have
\begin{align*}
&\E_X\left[
\E_{\substack{Q_1,Q_2\subseteq [k],\\|Q_1|=|Q_2|=q}}
\left[
\max\left\{\max_{\substack{R_1\subseteq Q_1, \\ R_2\subseteq Q_2}} \mc{U}_X(R_{2,X}^+, R_{1,X}^+),\max_{\substack{R_1\subseteq Q_1, \\ R_2\subseteq Q_2}} -\mc{U}_X(R_{2,X}^-,R_{1,X}^-)\right\}
\right]
\right]\\ \leq&
(k-q)^2\|U\|_\pcut + 4(k+q)q\|U\|_1
+ 4k^\va(k-1)\sqrt{k\ln k} \|U\|_1 \sqrt{ \frac{q p(\va+3)}{2}}
\left(1+k^{1-\va p}\frac{2^{p+3}\|U\|_p^p}{\|U\|_1^p (2p-3)}\right).
\end{align*}
\end{prop}

Combining these results (Proposition \ref{prop:upper_expect}) and optimising the choice of $q$, we end up with the following upper bound.

\begin{prop}\label{cor:q_sample}
   Let $k\geq 2$ be an integer and $p> 2$. Then for any kernel $U\in L^p_{sym}([0,1]^2)$, we have 
\begin{align*}
       & \E_{X} \left[\|\mc{U}_X\|_\pcut\right]-\|U\|_\pcut\leq \mc{C} k^{-\frac{1}{4}+\frac{1}{2p}}(\ln k)^{1/4}
\end{align*}
for some constant $\mc{C}$ depending only on $p$, $\|U\|_1$, $\|U\|_2$ and $\|U\|_p$.
\end{prop}

%%%%%%%%%%%%%%%%%%%
%
\section{Proof of Theorem \ref{Thm:First_sampling}}\label{section:proof_main}
%
%%%%%%%%%%%%%%%%%%%

%%%%%%%%%%%%%%%%%%%
%
\subsection{Upper bound}
%
%%%%%%%%%%%%%%%%%%%
Combining Proposition \ref{prop:upper_expect_truncation} with Lemma \ref{le:Azuma_lower}, we obtain the following high probability upper bound.
\begin{theorem}\label{thm:upper_bound} Let $k\geq 2$ be an integer, $p>2$, $\vb>0$ and let $U\in L^p_{sym}([0,1]^2)$ be a kernel. Then for any $\vf>0$ we have that the upper bound
\[
\|\mc{U}_X\|_\pcut-\|U\|_\pcut\leq \left(30\|U\|_p+
12\sqrt{\frac{\vf p}{2}}
      (1+\vb)
      \frac{\sqrt{\ln k}}{k^{\frac{p-3}{4p}-\vf}}\|U\|_1\right)
      k^{-\frac14+\frac1{4p}}
\]
holds with probability at least
\[
p_{\vb,\vf,p,U}:=1-\left(2+ 2\cdot\frac{2^p\|U\|_p^p}{\vb^p\|U\|_1^p}\right)k^{-\vf p}.
\]
\end{theorem}

\begin{proof}
  Let $\vf>0$ and let us apply Lemma \ref{le:Azuma_lower} with $\va:=\vf+1/p$ and $\lambda:=\sqrt{\frac{\va p-1}{2}\ln k}$.
  This yields
  \begin{align*}
     & \mathds{P} \left(\|\mc{U}_X\|_\pcut-\E[\|\mc{U}_X\|_\pcut]\geq 2 \sqrt{\frac{\va p-1}{2}\ln k}\cdot\frac{6}{k}
      \|U\|_1(1+\vb k^{\va})
      \sqrt{k}\right) \\ 
      & \qquad \qquad \leq 2e^{-2\lambda^2}+ 2 \cdot\frac{2^p\|U\|_p^p}{\vb^p\|U\|_1^p}k^{1-\va p} = \left(2+ 2\cdot \frac{2^p\|U\|_p^p}{\vb^p\|U\|_1^p}\right)k^{1-\va p},
  \end{align*}
where $1-\va p=-\vf p<0$. Now $\gamma>0$ and
\[
\gamma-\frac12=\vf+\frac1p-\frac12=-\frac14+\frac{1}{4p}-\left(\frac{p-3}{4p}-\vf\right),
\]
so this deviation from the expectation can be bounded above by
\[
12\sqrt{\frac{\vf p\ln k}2} \|U\|_1 (1+\vb)\frac{k^{-\frac14+\frac1{4p}}}{k^{\frac{p-3}{4p}-\vf}}.
\]
Together with Proposition \ref{prop:upper_expect_truncation} we thus obtain the desired result.
\end{proof}

Part (1) of Theorem \ref{Thm:First_sampling} follows by setting $\vb=\frac{2\|U\|_p}{\|U\|_1}$ (when $U$ is nontrivial).

\begin{remark}
Note that if $p>3$, then for $\vf\in (0,\frac{p-3}{4p})$ the upper bound is of order $k^{-\frac14+\frac{1}{4p}}$, and this is optimal here, with the bound getting progressively worse for larger values $\vf$. On the other hand, for $p\in[1,3]$, the dispersion term always dominates, and we actually need $p>2$ for the bound to converge to 0 in function of $k$.
\end{remark}

%%%%%%%%%%%%%%%%%%%
%
\subsection{Lower bound}
%
%%%%%%%%%%%%%%%%%%%

\begin{theorem}\label{thm:lower_bound}
Let $k\geq 2$ be an integer, $p> 2$, $U\in L^p_{sym}([0,1]^2)$ a kernel, $\vb>0$, and $\va\in(1/p,1/2)$.
Then with probability at least $p'_{\vb,\va,p,U}:=1- (2+2 \frac{2^p\|U\|_p^p }{\vb^p\|U\|_1^p})k^{1-\va p}$, we have
\[
\|\mc{U}_X\|_\pcut-\|U\|_\pcut \geq -\left(\|U\|_\pcut+6 \sqrt{2(\va p-1)}(1+\vb)\|U\|_1\right) k^{-1/2+\va}\sqrt{\ln k}.
\]
\end{theorem}

\begin{proof}
Let us apply Lemma \ref{le:Azuma_lower} with $\lambda:=\sqrt{\frac{\va p-1}{2}\ln k}$. This yields
\begin{align*}
       & \mathds{P} \left(\left|\Big.\|\mc{U}_X\|_\pcut - \E_X\left[ \|\mc{U}_X\|_\pcut\right]\right| \geq 2 \sqrt{\frac{\va p-1}{2}\ln k} \cdot\sqrt{k}\cdot\frac6k (1+\vb k^\va)\|U\|_1\right)\\
       & \qquad \qquad  \leq 2k^{1-\va p} + 2 k\frac{2^p\|U\|_p^p }{\vb^p k^{\va p}\|U\|_1^p}\leq\left(2+2 \frac{2^p\|U\|_p^p }{\vb^p\|U\|_1^p}\right)k^{1-\va p}=1-p'_{\vb,\va,p,U}.
\end{align*}
Combining with Lemma \ref{le:lower_expect}, we obtain that with probability at least
$p'_{\vb,\va,p,U}$, we have
\begin{align*}
        \|\mc{U}_X\|_\pcut-\|U\|_\pcut=& \left(\E_X[\|\mc{U}_X\|_\pcut]-\|U\|_\pcut\right)+\left(\|\mc{U}_X\|_\pcut-\E_X[\|\mc{U}_X\|_\pcut]\right)\\
\geq & -\frac1k\|U\|_\pcut- 2 \sqrt{\frac{\va p-1}{2}\ln k}\cdot \frac{6}{\sqrt{k}} (1+\vb k^\va)\|U\|_1\\
\geq&-\left(\|U\|_\pcut+6\sqrt{2(\va p-1)}(1+\vb)\|U\|_1\right) k^{-1/2+\va}\sqrt{\ln k}
.
\end{align*}
\end{proof}

Part (2) of Theorem \ref{Thm:First_sampling} follows by setting $\vb=\frac{2\|U\|_p}{\|U\|_1}$ (when $U$ is nontrivial).

\begin{remark}
Note that the original result in \cite{Borgs} only states a lower bound of order $O(k^{-1/4})$, but a slightly better optimized concentration inequality would actually yield $O(k^{-1/2})$.
\end{remark}

%%%%%%%%%%%%%%%%%%%
%
%
\section{Bounds for multigraphs and vector valued kernels}\label{Section:Banach}
%
%
%%%%%%%%%%%%%%%%%%%

Up to now, we have been treating unbounded, real valued kernels, that essentially correspond to limits of weighted graphs. In many applications, however, real valued kernels are not rich enough in structure to capture the limit objects. For edge-coloured graphs with at least 3 colours, or equivalently multigraphs, the adjacency matrix (or kernel) is actually vector valued, being pointwise equal to a (probability) distribution on the set of colours/multiplicities.

As shown by Lovász, Szegedy and the second author in \cite{KKLSz1}, in the most general left-convergence setting, kernels are represented by a function $U$ taking values in a dual Banach space, typically a space of measures. For instance, a classical graphon takes values $U(x,y)$ between 0 and 1, which can be interpreted as a probability distribution on the two-element set $\{0,1\}$, with weight $U(x,y)$ on 1. However, as soon as a probability distribution is over a set of more than 2 elements, a single value is not enough to describe it, and we have to resort to vector valued kernels. We shall here present a short overview of the formalism needed, and then show how the First Sampling Lemma can be adapted to this more general setting. The lower bound extends rather naturally, whilst the upper bound is presented in the finite dimensional setting, as an infinite dimensional result seems to be out of reach without a completely novel approach.

\subsection{Lower bound in dual spaces}

\begin{definition}
Let $Z$ be a separable Banach space, and $\mc{Z}$ its dual. A symmetric function $W:[0,1]^2\to\mc{Z}$ is called
an $L^p_\mc{Z}$-\emph{kernel} if
it is weak-* measurable, and the function $(x,y)\mapsto \|W(x,y)\|_\mc{Z}$
lies in $L^p\left([0,1]^2\right)$.
We define the norm
\[
\|W\|_\pcut := \sup_{\|f\|_{Z}=1}\bigl\|\langle f,W\rangle\bigr\|_\pcut=
\sup_{\|f\|_Z=1}\sup_{S,T\subseteq [0,1]}\int_{S\times T}\langle f,W\rangle.
\]
Moreover, $W$ is called a $\mc{Z}$-\emph{graphon}, if it is an $L^p_\mc{Z}$-kernel for all $p\in (1,\infty)$.
\end{definition}

For any $L^p_\mc{Z}$-\emph{kernel} $W$, we also define $\|W\|_p:=\left(\int_{[0,1]^2}\|W(x,y)\|_\mc{Z}^p\right)^{1/p}<\infty$ for any $p\geq1$.

We claim that the following result holds.

\begin{theorem}\label{thm:Banach_lower_bound}
Let $Z$ be a separable Banach space, and $\mc{Z}$ its dual. Let $k\geq 2$ be an integer, $p>2$, $U:[0,1]^2\to\mc{Z}$ an $L^p_\mc{Z}$-kernel, and $\va\in(1/p,1/2)$.
Then with probability at least $p'_{\va,p,U}:=1- 4k^{1-\va p}$, we have
\[
\|\mc{U}_X\|_\pcut-\|U\|_\pcut\geq -\left(\|U\|_\pcut+6\sqrt{2(\va p-1)}(\left\|U\|_1+2\|U\|_p\right)\right) k^{-1/2+\va}\sqrt{\ln k}.
\]

\end{theorem}

\begin{proof} Let $\varepsilon>0$ be arbitrary.
Choose an $f\in Z$ such that $\|f\|_Z=1$ and
\[
\|U\|_\pcut\leq (1+\varepsilon k^{-1/2+\va}\sqrt{\ln k})\|\langle f,U\rangle\|_\pcut.
\]
Then applying Theorem \ref{thm:lower_bound} to the  nontrivial kernel $\langle f,U\rangle$ with $\vb=2\|\langle f,U\rangle\|_p/\|\langle f,U\rangle\|_1\geq 2$, we obtain that with 
probability at least $p'_{\va,p,U}:=1- 4k^{1-\va p}$, we have
\begin{align*}
&\|\mc{U}_X\|_\pcut-\|U\|_\pcut\geq\|\langle f,\mc{U}_X\rangle\|_\pcut-(1+\varepsilon k^{-1/2+\va}\sqrt{\ln k})\|\langle f,U\rangle\|_\pcut\\
= & \|\langle f,\mc{U}_X\rangle\|_\pcut-\|\langle f,U\rangle\|_\pcut-\varepsilon k^{-1/2+\va}\sqrt{\ln k}\|\langle f,U\rangle\|_\pcut\\
\geq & -\left((1+\varepsilon)\|\langle f,U\rangle\|_\pcut+6\sqrt{2(\va p-1)}(\left\|\langle f,U\rangle\|_1+2\|\langle f,U\rangle\|_p\right)\right) k^{-1/2+\va}\sqrt{\ln k}\\
\geq& -\left((1+\varepsilon)\|U\|_\pcut+6\sqrt{2(\va p-1)}(\left\|U\|_1+2\|U\|_p\right)\right) k^{-1/2+\va}\sqrt{\ln k}.
\end{align*}
As the sets where this inequality fails are monotone increasing as $\varepsilon$ decreases, we obtain the desired result by passing to the limit.
\end{proof}

\subsection{Upper bound in $\R^n$}

The upper bound is unfortunately of a very different nature, and the real valued case does not readily transfer to the multidimensional case in the way it did for the lower bound. Indeed, here the choice of $f$ has to be tailored to $\mc{U}_X$. In finite dimensions, compactness still allows us to achieve a sampling lemma, but with the size of the bad set increasing exponentially with the dimension. To keep in line with the typical application of $\mc{Z}$ being a space of signed measures, we formulate the next result for $\R^n$ equipped with the 1-norm.

\begin{theorem}\label{thm:Banach_upper_bound} Let $k, n\geq2$ be integers, $p>3$, and $\vf>0$. Let further $\mc{Z}=(\R^n,\|\cdot\|_1)$ and let $U:[0,1]^2\to\mc{Z}$ be an $L^p_{\mc{Z}}$-kernel. Then we have that with probability at least
\[
1-2^{n+2}nk^{-\vf p+\frac{(n-1)(p-1)}{4p}},
\]
the upper bound
\[
\|\mc{U}_X\|_\pcut-\|U\|_\pcut\leq 10\left[\|U\|_\pcut+ 30\|U\|_p+
6\sqrt{2\vf p}
      (\|U\|_1+2\|U\|_p)
      \frac{\sqrt{\ln k}}{k^{\frac{p-3}{4p}-\vf}}
      \right]k^{-\frac14+\frac1{4p}}
\]
is satisfied.
\end{theorem}
\begin{remark}
Note that the size of the bad set and the upper bound will both tend to 0 if $\frac{p-1}p\cdot\frac{n-1}{4p}<\frac{p-3}{4p}$ and $\vf\in \left(\frac{p-1}p\cdot\frac{n-1}{4p},\frac{p-3}{4p}\right)$. We can thus find an appropriate $\vf$ as soon as $p>\frac{n+2+\sqrt{n^2+8}}{2}$, in particular whenever $p>n+2$.
\end{remark}
\begin{proof}
Given an $\epsilon\in(0,1)$, one can choose a finite subset $E$ of the unit sphere $\Gamma$ in $Z=(\R^n,\|\cdot\|_\infty)$ of size at most $2n\cdot(\frac2\epsilon)^{n-1}$ such that for any $v\in \Gamma$ there exists an $f\in E$ with $\|v-f\|_\infty\leq\epsilon$. Thus for any $X$ we have
\[
\|\mc{U}_X\|_\pcut\leq (1-\epsilon)^{-1}\max_{f\in E} \|\langle f,\mc{U}_X\rangle\|_\pcut.
\]
Applying Theorem \ref{thm:upper_bound} to the nontrivial kernel $\langle f,U\rangle$ with $\vb=2\|\langle f,U\rangle\|_p/\|\langle f,U\rangle\|_1>0$, we obtain that for each $f\in E$, with 
probability at least $1-4k^{-\vf p}$, we have
\begin{align*}
\|\langle f,\mc{U}_X\rangle\|_\pcut\leq&\|\langle f,U\rangle\|_\pcut+ \left(30\|\langle f,U\rangle\|_p+
6\sqrt{2\vf p}
      (\|\langle f,U\rangle\|_1+2\|\langle f,U\rangle\|_p)
      \frac{\sqrt{\ln k}}{k^{\frac{p-3}{4p}-\vf}}\right)
      k^{-\frac14+\frac1{4p}}\\
      \leq &\|U\|_\pcut+ \left(30\|U\|_p+
6\sqrt{2\vf p}
      (\|U\|_1+2\|U\|_p)
      \frac{\sqrt{\ln k}}{k^{\frac{p-3}{4p}-\vf}}\right)
      k^{-\frac14+\frac1{4p}}.
\end{align*}
By the union bound, this holds for all $f\in E$ simultaneously with probability at least $1-8n(\frac2\epsilon)^{n-1}k^{-\vf p}$, leading to
\begin{align*}
\|\mc{U}_X\|_\pcut\leq (1-\epsilon)^{-1}\left[\|U\|_\pcut+ \left(30\|U\|_p+
6\sqrt{2\vf p}
      (\|U\|_1+2\|U\|_p)
      \frac{\sqrt{\ln k}}{k^{\frac{p-3}{4p}-\vf}}\right)
      k^{-\frac14+\frac1{4p}}\right].
\end{align*}
Letting $\epsilon:=k^{-\frac14+\frac1{4p}}<1/\sqrt[6]{2}<9/10$, we have $(1-\epsilon)^{-1}<1+10\epsilon<10$, hence with probability at least
$1-8n(2k^{\frac14-\frac1{4p}})^{n-1}k^{-\vf p}=1-2^{n+2}nk^{-\vf p+\frac{(n-1)(p-1)}{4p}}$, we have
\begin{align*}
\|\mc{U}_X\|_\pcut-\|U\|_\pcut\leq 10\left[\|U\|_\pcut+ 30\|U\|_p+
6\sqrt{2\vf p}
      (\|U\|_1+2\|U\|_p)
      \frac{\sqrt{\ln k}}{k^{\frac{p-3}{4p}-\vf}}
      \right]k^{-\frac14+\frac1{4p}}.
\end{align*}
\end{proof}

Note that if instead of applying the real-valued results, one aims to generalize their proof method (sketched for the unbounded kernel case in the Appendix, Section \ref{sect:Appendix}), it quickly turns out that the attempts will fail at several key points.

In the vector valued case, there is no complete order, and hence no natural ``positive'' or ``negative'' contribution to $\mc{U}_X(R_1,R_2)$ that would allow for a variation on Lemma \ref{Lemma:BS12}.

Furthermore, Banach space valued concentration results extending Azuma's inequality with usable bounds are generally confined to uniformly smooth spaces (cf. \cite{Luo, Naor}), which are automatically reflexive. As mentioned earlier, the most relevant applications are where $\mc{Z}$ is a space of signed measures (with total variation norm), which is not reflexive when infinite dimensional. Also, even in the finite dimensional case the space $(\R^n,\|\cdot\|_1)$ is not uniformly smooth.

%%%%%%%%%%%%%%%%%%%
%
%
\section{Second Sampling Lemma}\label{sect:second}
%
%
%%%%%%%%%%%%%%%%%%%

Our goal here is to directly bound the cut distance between the function $U$ and its sample $\mc{U}_X$. We will be combining our generalized First Sampling Lemma with the original bounded kernel result below.

\begin{prop}[Theorem 4.7 (i), \cite{Borgs}]\label{prop:second_bounded}
Let $U\in L^\infty_{sym}([0,1]^2)$ be a bounded kernel, and $k\geq 2$ an integer. Then 
\[
\delta_\pcut(U,\mc{U}_k)\leq \frac{20}{\sqrt{\log_2 k}}\|U\|_\infty
\]
holds with probability at least $1-e^{\frac{-k^2}{2\log_2 k}}$.
\end{prop}

Similarly to our approach in Section \ref{sect:truncation}, we want to truncate $U$ at large absolute values, and bound the errors in function of the truncation value $f(k)$.
So for each $k$, let $U^*_k$ denote a truncated version of $U$, again.

Then we have the following.

\begin{theorem}\label{thm:second_sampling}
Let $p>2$, and $\vf\in(0,\frac12-\frac1p)$. Then there exists a constant $C_{\vf,p}$ such that for any integer $k\geq 2$ and any kernel $U\in L^p_{sym}([0,1]^2)$, the inequality 
\[
\delta_\pcut(U,\mc{U}_X)\leq  C_{\vf,p} (\|U\|_1+\|U\|_p)(\ln k)^{-\frac12+\frac1{2p}}
\]
holds with probability at least $p_{k,\vf,p}:=1-e^{\frac{-k^2}{2\log_2 k}}-4k^{-\vf p}$.
\end{theorem}

\begin{proof}
Let $\mc{U}_{k,X}^*$ denote the random $k$-sample obtained from $U_k^*$.
Then we have the following.
\begin{align*}
\delta_\pcut(U,\mc{U}_X)\leq & \delta_\pcut(U,U_k^*)+\delta_\pcut(U_k^*,\mc{U}_{k,X}^*)+\delta_\pcut(\mc{U}_{k,X}^*,\mc{U}_X)\\
\leq & \|U-U_k^*\|_\pcut +\delta_\pcut(U_k^*,\mc{U}_{k,X}^*)+\|\mc{U}_{k,X}^*-\mc{U}_X\|_\pcut.
\end{align*}

By the bounded kernel result Proposition \ref{prop:second_bounded}, we have that
\[
\delta_\pcut(U_k^*,\mc{U}_{k,X}^*)\leq \frac{20}{\sqrt{\log_2 k}}\|U_k^*\|_\infty\leq
\frac{20}{\sqrt{\log_2 k}}f(k)
\]
holds with probability at least $1-e^{\frac{-k^2}{2\log_2 k}}$.

On the other hand, since $\mc{U}_{k,X}^*-\mc{U}_X$ is the random $k$-sample of $U_k^*-U$, the upper bound from Theorem \ref{Thm:First_sampling} tells us that with probability at least $1-4k^{-\vf p}$,
\begin{align*}
&\|\mc{U}_{k,X}^*-\mc{U}_X\|_\pcut\\
& \leq \|U_k^*-U\|_\pcut+ 30\|U-U_k^*\|_p k^{-\frac14+\frac1{4p}}+
6\sqrt{2\vf p}(\|U-U_k^*\|_1+2\|U-U_k^*\|_p)\sqrt{\ln k}\,k^{-\frac12+\frac1{p}+\vf}.
\end{align*}
By assumption, we have $\|U-U_k^*\|_1\leq \|U\|_1$ and $\|U-U_k^*\|_p\leq\|U\|_p$, hence with $C_{1,\vf,p,U}:=30\|U\|_p$ and $C_{2,\vf,p,U}:=6\sqrt{2\vf p}(\|U-U_k^*\|_1+2\|U-U_k^*\|_p)\leq 6\sqrt{2\vf p}(\|U\|_1+2\|U\|_p)$, we have that
\[
\|\mc{U}_{k,X}^*-\mc{U}_X\|_\pcut\leq \|U_k^*-U\|_\pcut+C_{1,\vf,p,U} k^{-\frac14+\frac1{4p}}+C_{2,\vf,p,U}\sqrt{\ln k}\,k^{-\frac12+\frac1{p}+\vf}
\]
holds with probability at least $1-4k^{-\vf p}$.

Finally, by the proof of Proposition \ref{prop:upper_expect_truncation},
\begin{align*}
\|U_k^*-U\|_\pcut\leq\|U_k^*-U\|_1\leq 2^p\|U\|_p^p\frac{f(k)}{(f(k)-\|U\|_1)^p}.
\end{align*}

Thus with probability at least $p_{k,\vf,p}=1-e^{\frac{-k^2}{2\log_2 k}}-4k^{-\vf p}$, we have
\begin{align*}
\delta_\pcut(U,\mc{U}_X)\leq & \|U-U_k^*\|_\pcut +\delta_\pcut(U_k^*,\mc{U}_{k,X}^*)+\|\mc{U}_{k,X}^*-\mc{U}_X\|_\pcut\\
\leq & 2^{p+1}\|U\|_p^p\frac{f(k)}{(f(k)-\|U\|_1)^p}+C_{1,\vf,p,U} k^{-\frac14+\frac1{4p}}+C_{2,\vf,p,U}\sqrt{\ln k}\,k^{-\frac12+\frac1{p}+\vf}\\
&+\frac{20}{\sqrt{\log_2 k}}f(k).
\end{align*}
Balancing the first and last term yields an optimal choice of $f(k)=c(\ln k)^{1/2p}$ and an order of magnitude $O((\ln k)^{-\frac12+\frac1{2p}})$. For any $\vf\in(0,\frac12-\frac1p)$, the two middle terms are then of lower order (polynomial in $k$), and setting $c=3\|U\|_p$ as in the proof of Proposition \ref{prop:upper_expect_truncation}, we can find a constant $C_{\vf,p}$ such that with probability at least $p_{k,\vf,p}$, the inequality
\[
\delta_\pcut(U,\mc{U}_X)\leq C_{\vf,p} (\|U\|_1+\|U\|_p)(\ln k)^{-\frac12+\frac1{2p}}
\]
holds, as desired.

\end{proof}

Theorem \ref{Thm:Second_sampling} is now a simple corollary of this result. Also, we can show that almost sure convergence of samples holds not only in the metric of left-convergence, but also in the cut metric $\delta_\pcut$, provided the kernel $U$ satisfies $U\in L^p$ for some $p>4$.

\begin{cor*}(Corollary \ref{cor:almost_sure_left}) 
Let $p>4$, and for each integer $k\geq2$, let $X^{(k)}\in[0,1]^k$ be a uniform random $k$-vector. Then for any kernel $U\in L^p_{sym}([0,1]^2)$, the convergence $\delta_\pcut(U,\mc{U}_{X^{(k)}})\to0$ holds with probability 1.
\end{cor*}
\begin{proof}
Note that when $p>4$, we may in Theorem \ref{thm:second_sampling} choose $\vf=\varepsilon+1/p$ with $\varepsilon\in(0, 1/2-2/p)$, meaning that $1-p_{k\vf,p}=O(k^{-1-\varepsilon p})$. By the Borel-Cantelli lemma, with probability 1 the inequality
\[
\delta_\pcut(U,\mc{U}_{X^{(k)}})\leq C_{\vf,p} (\|U\|_1+\|U\|_p)(\ln k)^{-\frac12+\frac1{2p}}
\]
holds for all but finitely many integers $k\geq2$, implying the desired convergence.
\end{proof}

\begin{remark}
By the same arguments, we also see that for $p\in(2,4]$, the sequence $\mc{U}_X$ almost surely has $U$ as a $\delta_\pcut$-accumulation point.
\end{remark}

Finally, we note that the vector valued version of the Second Sampling Lemma cannot be directly derived from its real valued counterpart the way it was done for the First Sampling Lemma.
Let $\Gamma$ denote the unit sphere of the Banach space $Z$, and let $U$ be an $L^p_\mc{Z}$-kernel. We then have

\[
\delta_\pcut(\mc{U}_X,U)=\inf_{\psi\in\Psi} \|\mc{U}_X-U\circ\psi\|_\pcut=
\inf_{\psi\in\Psi} \sup_{f\in\Gamma}\|\langle f,\mc{U}_X-U\circ\psi\rangle\|_\pcut.
\]

The issue is that to apply the real valued case, we would have to compare this infimum with the values
\[
\delta_\pcut(\langle f,\mc{U}_X\rangle,\langle f,U\rangle)=\inf_{\psi\in\Psi}\|\langle f,\mc{U}_X-U\circ\psi\rangle\|_\pcut,
\]
which in essence would amount to bounding $\inf_{\psi\in\Psi} \sup_{f\in\Gamma}$ by $\sup_{f\in\Gamma}\inf_{\psi\in\Psi}$, which clearly is not feasible.

%%%%%%%%%%%%%%%%%%%
%
%
\section*{Acknowledgement.}
%
%
%%%%%%%%%%%%%%%%%%%
Panna Tímea Fekete's Project No.\ 1016492.\ has been implemented with the support provided by the Ministry of Culture and Innovation of Hungary from the National Research, Development and Innovation Fund, financed under the KDP-2020 funding scheme. The authors were also supported by ERC Synergy Grant No.~810115.

\section{Appendix}\label{sect:Appendix}

Here we provide the proofs and intermediate results for the random $q$-subsample approach described in Section \ref{subsect:original}.

Recall that $\mc{N}_{(R_1,R_2),\,b}(X):=\E_{X_{[k]\setminus([b-1]\cup Q)}}\left[\mc{U}_X(R_{2,X}^+,R_{1,X}^+)\right]$ and that given an integer $a\in [k]$, $Y^{(a)}$ denotes the random $k$-vector obtained by replacing the $a$-th term $X_a$ in $X$ by $X_0$, i.e., $Y^{(a)}=(X_1,\ldots,X_{a-1},X_0,X_{a+1},\ldots, X_k)$.

\begin{lemma}\label{le:martingale_diff2}
Let $U\in L^1_{sym}([0,1]^2)$, $k\geq 1$ an integer, and $X=(X_1, X_2, \dots, X_k) \in [0,1]^k$ a uniform random $k$-vector. Let $q\in [k]$ be an integer, $Q_1, Q_2$ two independent uniform random $q$-subsets of $[k]$
and $a\in [k]\setminus Q$, where $Q:=Q_1\cup Q_2$. Then for any $R_1\subseteq Q_1$ and $R_2\subseteq Q_2$ we have the following.

\begin{enumerate}
\item[($\mc{A}$)]
$\E_{X_{[k]\setminus Q}}\left[\mc{U}_X\left(R_{2,X}^+,R_{1,X}^+\right) \right] \leq  (k-|Q|)^2  \|U\|^+_\pcut  + 2k \sum\limits_{i \in Q} U_{X_i}+\mc{U}_X(Q,Q)$.
\item[($\mc{B}$)] Almost surely
\begin{align*}
    \left|\mc{N}_{(R_1,R_2),\,a+1}(X)-\mc{N}_{(R_1,R_2),\,a+1} \left(Y^{(a)}\right)\right| 
         \leq &\sum_{j\in [k]\setminus \{a\}}
        \left(\big| U(X_a,X_j) \big| + \big|U(X_0,X_j)\big| \right) \\
        &+
        (k-1)\left(U_{X_a} + U_{X_0}\right).             
\end{align*}
\item[($\mc{C}$)] Almost surely
\begin{align*}
    \left|\mc{N}_{(R_1,R_2),\,a}(X)-\mc{N}_{(R_1,R_2),\,a+1}(X)\right|\leq &
    \sum_{j\in [k]\setminus \{a\}}
    \left(U_{X_j} + 
    \big| U(X_a,X_j)\big|\right) + (k-1)
    \left(\|U\|_1 +
    U_{X_a}\right).
\end{align*}
\end{enumerate}
\end{lemma}

\begin{proof}

\underline{Part $(\mc{A})$:}
Define
\begin{equation*}
    \begin{split}
        \mc{R}_1 = & \left\{y \in [0,1] : \sum_{i\in R_1} U(X_i, y) > 0\right\},\\
        \mc{R}_2 = & \left\{y \in [0,1] : \sum_{i\in R_2} U(y, X_i) > 0\right\}.
    \end{split}
\end{equation*}

For every \(i \in [k] \setminus Q\) and \(j \in [k] \setminus Q\), the contribution of a term \(U(X_i, X_j )\) to the value of \(\E_{X_{[k]\setminus Q}}\left[\mc{U}_X(R_{2,X}^+, R_{1,X}^+) \right]\) is
\[\int_{\mc{R}_2\times \mc{R}_1} U(x,y) dx dy \leq \|U\|^+_\pcut.\]

The contribution of the terms \(U(X_i, X_j )\) with exactly one of $i$ and $j$ in $Q$ 
adds up to at most 
$ 2 k \sum_{i \in Q} U_{X_i}$,
because for example assuming \(i\in Q, j \in [k]\setminus Q\) and using independence,
\begin{equation*}
    \begin{split}
        \E_{X_{[k]\setminus Q}}\left[U(X_i,X_j)\right] & = 
        \int_{[0,1]^{k-|Q|}} U(X_i,X_j) \prod_{\alpha\in [k] \setminus Q} dX_\alpha \\
        & =  \int_0^1 U(X_i,X_j)dX_j
        \leq \int_0^1 \Big|U(X_i,X_j)\Big| dX_j =  U_{X_i}.
    \end{split}
\end{equation*}
Hence
\begin{align*}
\E_{X_{[k]\setminus Q}}\left[\mc{U}_X\left(R_{2,X}^+, R_{1,X}^+\right) \right] \leq& (k-|Q|)^2  \|U\|^+_\pcut  + 2k \sum_{i \in Q} U_{X_i}+\sum_{i\in R_{2,X}^+\cap Q,\,j\in R_{1,X}^+\cap Q} U(X_i,X_j)\\
\leq&  (k-|Q|)^2  \|U\|^+_\pcut  + 2k \sum_{i \in Q} U_{X_i}+\mc{U}_X(Q,Q).
\end{align*}

\underline{Part $(\mc{B})$:}
We wish to bound the change to the value of \(\mc{U}_X(R_{2,X}^+, R_{1,X}^+)\) as $X$ is replaced by $Y^{(a)}$.
The function \(\mc{U}_X\left(R_{2,X}^+,R_{1,X}^+\right)\) is a function of the independent random variables \(\left(X_\ell\right)_{\ell \in [k] \setminus Q}\) (as we fixed the other values), of which exactly one changes to $X_0$ in $Y^{(a)}$ since $a\notin Q$. This change has a double effect. On the one hand, it affects the values of $U$ that are considered, and on the other hand, as a secondary effect, the sets $R_{2,\cdot}^+$ and $R_{1,\cdot}^+$ themselves may change.

Now, since $X_i$ is fixed for all $i\in R_1\cup R_2\subseteq Q$ but $a\notin Q$, we actually have both $R_{2,X}^+\Delta R_{2,Y^{(a)}}^+\subseteq \{a\}$ and $R_{1,X}^+\Delta R_{1,Y^{(a)}}^+\subseteq \{a\}$. Therefore the only terms that may change are the ones in the $a$-th row or column of $\mc{U}_X$ and $\mc{U}_{Y^{(a)}}$, and so
\begin{equation*}
    \begin{split}
        & \left|\mc{U}_{X}(R_{2,X}^+, R_{1,X}^+)-\mc{U}_{Y^{(a)}}(R_{2,Y^{(a)}}^+, R_{1,Y^{(a)}}^+)\right| \leq \sum_{j\in [k]\setminus \{a\}} \big| U(X_a,X_j) \big| + \sum_{j\in [k]\setminus \{a\}} \big|U(Y^{(a)}_a,Y^{(a)}_j)\big| \\
        & \quad = \sum_{j\neq a,\, j \in [k]\setminus Q} \big| U(X_a,X_j) \big|+ \sum_{ j\in Q} \big| U(X_a,X_j) \big| + \sum_{j\neq a,\, j \in [k]\setminus Q} \big|U(X_0,X_j)\big| + \sum_{j \in Q} \big|U(X_0,X_j)\big|,
    \end{split}
\end{equation*}
which then leads to
\begin{align*}
        & \left|\mc{N}_{(R_1,R_2),\,a+1}(X)-\mc{N}_{(R_1,R_2),\,a+1} \left(Y^{(a)}\right)\right|
        \\
        \leq & \E_{X_{[k]\setminus ([a] \cup Q)}} \left[ \left| \mc{U}_X\left(R_{2,X}^+, R_{1,X}^+\right) - \mc{U}_{Y^{(a)}} \left(R_{2,Y^{(a)}}^+, R_{1,Y^{(a)}}^+\right) \right| \right] \\
        \leq & \E_{X_{[k]\setminus ([a] \cup Q)}}\left[\sum_{j\neq a,\, j \in [k]\setminus Q} \big| U(X_a,X_j) \big| \right]+ \E_{X_{[k]\setminus ([a] \cup Q)}}\left[\sum_{ j\in Q} \big| U(X_a,X_j) \big| \right]\\
        & + \E_{X_{[k]\setminus ([a] \cup Q)}}\left[\sum_{j\neq a,\, j \in [k]\setminus Q} \big|U(X_0,X_j)\big| \right]+ \E_{X_{[k]\setminus ([a] \cup Q)}}\left[\sum_{j \in Q} \big|U(X_0,X_j)\big|\right] \\
         = &\E_{X_{[k]\setminus ([a] \cup Q)}}\left[\sum_{j\neq a,\, j \in [k]\setminus Q} \left(\big| U(X_a,X_j) \big| + \big|U(X_0,X_j)\big| \right)\right]\\
        & + \E_{X_{[k]\setminus ([a] \cup Q)}}\left[\sum_{j \in Q} \left(\big|U(X_a,X_j)\big| + \big| U(X_0,X_j) \big| \right)\right] \\
        = &
        \sum_{j\in [a-1]\setminus Q} 
        \left(\big| U(X_a,X_j) \big| + \big|U(X_0,X_j)\big| \right) +\sum_{j \in [k]\setminus ([a] \cup Q)}
        \E_{X_{[k] \setminus ([a] \cup Q)}}\left[\big| U(X_a,X_j) \big| + \big|U(X_0,X_j)\big| \right]\\
        &  + \sum_{j \in Q} \left(\big|U(X_a,X_j)\big| + \big| U(X_0,X_j) \big| \right) \\ 
         =  &\sum_{j \in [a-1]\setminus Q}
         \left(\big| U(X_a,X_j) \big| + \big|U(X_0,X_j)\big| \right) +\sum_{j \in [k]\setminus ([a] \cup Q)}
         \left(U_{X_a} + U_{X_0}\right)\\
        &  + \sum_{j \in Q} \left(\big|U(X_a,X_j)\big| + \big| U(X_0,X_j) \big| \right)\\
        \leq&\sum_{j\in [k]\setminus \{a\}}
        \left(\big| U(X_a,X_j) \big| + \big|U(X_0,X_j)\big| \right) + (k-1)\left(U_{X_a} + U_{X_0}
        \right).         
\end{align*}

\underline{Part $(\mc{C})$:}
Finally, we bound the martingale difference.
\begin{align*}
&\left|\mc{N}_{(R_1,R_2),\,a}(X)-\mc{N}_{(R_1,R_2),\,a+1}(X)\right| \\
& = \Big|\E_{X_{[k] \setminus([a-1] \cup Q)}} \left[ \mc{U}_X\left(R_{2,X}^+, R_{1,X}^+\right)\right] - \E_{X_{[k]\setminus([a] \cup Q)}} \left[\mc{U}_{X} \left(R_{2,X}^+, R_{1,X}^+\right) \right] \Big| \\
  & = \Big|\E_{X_0,X_{[k] \setminus([a] \cup Q)}
  } \left[ \mc{U}_{Y^{(a)}}\left(R_{2,Y^{(a)}}^+, R_{1,Y^{(a)}}^+\right)\right] - \E_{X_{[k] \setminus([a] \cup Q)}
  } \left[\mc{U}_{X} \left(R_{2,X}^+, R_{1,X}^+\right) \right] \Big| \\
   &= \left|\E_{X_0}  \left[ \E_{X_{[k] \setminus([a] \cup Q)}
  } \left[ \mc{U}_{Y^{(a)}}\left(R_{2,Y^{(a)}}^+, R_{1,Y^{(a)}}^+\right) - \mc{U}_{X} \left(R_{2,X}^+, R_{1,X}^+\right) \right]\right]\right|\\
  &\leq \E_{X_0}  \left| \E_{X_{[k] \setminus([a] \cup Q)}
  } \left[\mc{U}_{Y^{(a)}}\left(R_{2,Y^{(a)}}^+, R_{1,Y^{(a)}}^+\right) - \mc{U}_{X} \left(R_{2,X}^+, R_{1,X}^+\right) \right]\right| \\
  & \stackrel{(\mc{B})}{\leq} 
  \E_{X_0} \left[\sum_{j\in [k]\setminus \{a\}}
        \left(\big| U(X_a,X_j) \big| + \big|U(X_0,X_j)\big| \right)\right] + \E_{X_0} \left[ (k-1)\left(U_{X_a} + U_{X_0}\right)\right] \\
  & =\sum_{j\in [k]\setminus \{a\}}
    \left(U_{X_j} + 
    \big| U(X_a,X_j)\big|\right) + (k-1)
    \left(\|U\|_1 +
    U_{X_a}\right).
\end{align*}
\end{proof}

As in our approach to the dispersion term, this now allows us to bound the size of the sets where the martingale difference is large, and apply Lemma \ref{Lemma:generalised_azuma}. However, for what will follow, we need to keep track of our actual exceptional sets, not only their measure.

\begin{lemma}\label{le:Azuma_upper}
Let $U\in L^1_{sym}([0,1]^2)$ be a kernel, $k\geq 2$ an integer, $q\in[k]$, $\vb,\va>0$, $Q_1,Q_2\subseteq [k]$ with $|Q_1|=|Q_2|=q$, and $\beta:=4(k-1) \|U\|_1 \left(1+ \vb k^\va\right)$. Let further
\[
H_{\beta,a,R_1,R_2}:= \left\{X\in [0,1]^k \quad \bigg| \quad \left|\mc{N}_{(R_1,R_2),\,a}(X)-\mc{N}_{(R_1,R_2),\,a+1}(X)\right| \geq \beta\right\}
\]
for $a\in [k]$ and $R_1\subseteq Q_1$, $R_2\subseteq Q_2$. Then
\[
\bigcup_{\substack{R_1\subseteq Q_1,\\ R_2\subseteq Q_2}} \quad \bigcup_{a\in[k]} H_{\beta,a,R_1,R_2} \subseteq (L_{\vb,\va}^0)^c.
\]
Consequently, letting $X=(X_1, X_2, \dots, X_k) \in [0,1]^k$ be a uniform random $k$-vector, for any $\lambda>0$
 we have
\begin{align*}
  &      \mathds{P}_X \left((X\in L_{\vb,\va}^0)\wedge\left(\max_{\substack{R_1\subseteq Q_1,\\ R_2\subseteq Q_2}}\left|\Big. \mc{U}_X(R_{2,X}^+,R_{1,X}^+)-\E_{X_{[k]\setminus Q}}\left[\mc{U}_X(R_{2,X}^+,R_{1,X}^+)\right]\right|
         \geq 2 \lambda \sqrt{k\cdot  \beta^2}\right)\right)\\
         \leq & 2\cdot 4^q e^{-2 \lambda^2}.
\end{align*}
\end{lemma}

\begin{proof}
By Lemma \ref{le:martingale_diff2}, we have for any $X\in L_{\vb,\va}^0$, $R_1\subseteq Q_1$, $R_2\subseteq Q_2$ and $a\in[k]\setminus Q$ that
\begin{align*}
\left|\mc{N}_{(R_1,R_2),\,a}(X)-\mc{N}_{(R_1,R_2),\,a+1}(X)\right|
\leq& \sum_{j\in [k]\setminus \{a\}}
    \left(U_{X_j} + 
    \big| U(X_a,X_j)\big|\right) + (k-1)
    \left(\|U\|_1 +
    U_{X_a}\right)\\
    < & 4(k-1) \|U\|_1 \left(1+ \vb k^\va\right)=\beta.
\end{align*}

Note that \(\beta\) depends only on $U$, $k$, $\vb$ and $\va$, and not on any of \(a\), $R_1$ and $R_2$. Also, whenever $a\in Q$, we actually have $\mc{N}_{(R_1,R_2),\,a}(X)=\mc{N}_{(R_1,R_2),\,a+1}(X)$, hence indeed
\[
\bigcup_{\substack{R_1\subseteq Q_1,\\ R_2\subseteq Q_2}} \quad \bigcup_{a\in[k]} H_{\beta,a,R_1,R_2} \subseteq (L_{\vb,\va}^0)^c.
\]
By applying Lemma \ref{Lemma:generalised_azuma} to the $4^q$ functions $f_{R_1,R_2}(X):=\mc{U}_X(R_{2,X}^+,R_{1,X}^+): [0,1]^k\to\mathbb{R}$ (where $R_1\subseteq Q_1$ and $R_2 \subseteq Q_2$), we obtain the desired result.
\end{proof}

Next note that the set $L_{\vb,\va}^0$ is the same for $U$ and $-U$, so the above also yields a corresponding concentration result for the $4^q$ functions $f_{R_1,R_2}(X):=-\mc{U}_X(R_{2,X}^-,R_{1,X}^-): [0,1]^k\to\mathbb{R}$. In particular, we obtain the following result.

\begin{cor*}[Corollary \ref{cor:Azuma_upper}]
Let $k\geq 2$, $q\in[k]$ be integers, $p\geq 2$, $\vb\geq1$, $\va>0$, $Q_1,Q_2\subseteq [k]$ with $|Q_1|=|Q_2|=q$. Then for any kernel $U\in L^p_{sym}([0,1]^2)$, there exists a set $S_{k,\vb,\va,Q_1,Q_2}$ of measure at most $\frac{4}{\vb^p k^{\va p}}$ such that for any $X\in L_{\vb,\va}^0\setminus S_{k,\vb,\va,Q_1,Q_2}$ and $*\in\left\{+,-\right\}$, we have
\begin{align*}
\max_{\substack{R_1\subseteq Q_1, \\ R_2\subseteq Q_2}}\left|\mc{U}_X(R_{2,X}^*,R_{1,X}^*) - \E_{X_{[k]\setminus Q}}\left[\mc{U}_X(R_{2,X}^*,R_{1,X}^*)\right]\right|
\leq 8(k-1)\sqrt{k\ln k} \|U\|_1\left(1 + \vb k^\va\right)\sqrt{ \frac{q\vb p(\va+3)}{2}}.
\end{align*}

\end{cor*}
\begin{proof}

Let \(\lambda = \sqrt{\frac{q\vb (\va+3) p\ln k}{2}}\). Then
\[
2\lambda\sqrt{k\beta^2}=2\lambda\sqrt{k}\left(4(k-1)\|U\|_1(1+\vb k^\va)\right)
=8(k-1)\sqrt{k\ln k} \|U\|_1 \left(1+ \vb k^\va\right)\sqrt{ \frac{q\vb p(\va+3)}{2}}.
\]
Further, since $k\geq 2$, we have $k^{2\vb}\geq \vb$, but we also have $\vb^q\geq \vb$, $q\vb\geq 1$ and $\vb p\geq2$ as $q,\vb\geq 1$ and $p\geq 2$, implying $k^{\vb p}\geq4$. Thus
\begin{equation*}
    \begin{split}
        e^{-2 \lambda^2}
        & = e^{-q\vb (\va+3) p\ln k} = k^{-q \vb \va p}\cdot k^{-2q \vb p} \cdot k^{-q \vb p} \leq k^{-\va p} \cdot \vb^{-p} \cdot 4^{-q}.
    \end{split}
\end{equation*}

Applying Lemma \ref{le:Azuma_upper} to $U$ and $-U$ with our choice of $\lambda$ and setting
\begin{align*}
&
S_{k,\vb,\va,Q_1,Q_2} :=\\
&\bigcap_{ *\in\left\{+,-\right\} }\left\{
X\in L_{\vb,\va}^0:\left(\max_{\substack{R_1\subseteq Q_1,\\ R_2\subseteq Q_2}}\left|\Big. \mc{U}_X(R_{2,X}^*,R_{1,X}^*)-\E_{X_{[k]\setminus Q}}\left[\mc{U}_X(R_{2,X}^*,R_{1,X}^*)\right]\right|
         \geq 2 \lambda \sqrt{k\cdot  \beta^2}\right)\right\}
\end{align*}
then finishes the proof.
\end{proof}

The above proof actually implies
\begin{equation*}
\mathds{P}_X(X\not\in L_{\vb,\va}^0\setminus S_{k,\vb,\va,Q_1,Q_2})\leq 2k\frac{2^p\|U\|_p^p }{\vb^p k^{\va p}\|U\|_1^p}+\frac{4}{\vb^p k^{\va p}}\leq\frac43\cdot\frac{2^{p+1}k\|U\|_p^p}{\vb^p k^{\va p}\|U\|_1^p}.
\end{equation*}
Now define the following measurable function on $[0,1]^k$:
\[
N_{k,\va,Q_1,Q_2}(X):=\inf\left\{
\vb\geq1
\left|
X\in L_{\vb,\va}^0\setminus S_{k,\vb,\va,Q_1,Q_2}
\right.
\right\}.
\]
Using the notation $\kappa_{p,\va}=\frac{2^{p+3}k\|U\|_p^p}{3k^{\va p}\|U\|_1^p}$, we then have for any $t\geq 1$ that
\begin{equation}\label{eq:N_tail}
\mathds{P}_X\left(N_{k,\va,Q_1,Q_2}(X)>t\right)\leq\mathds{P}_X\left(X\not\in L_{t,\va}^0\setminus S_{k,t,\va,Q_1,Q_2}\right)\leq \kappa_{p,\va}t^{-p}.
\end{equation}
Combining the previous result with part ($\mc{A}$) of Lemma \ref{le:martingale_diff2}, this allows us to prove the following upper bound on $\max\left\{\max\limits_{R_i\subseteq Q_i} \mc{U}_X(R_{2,X}^+, R_{1,X}^+),\max\limits_{R_i\subseteq Q_i} -\mc{U}_X(R_{2,X}^-,R_{1,X}^-)\right\}$.

\begin{lemma}\label{le:max_upper}
Let $k\geq 2$ and $q\in[k]$ be integers, $p\geq 2$, $\va>0$, $Q_1,Q_2\subseteq [k]$ with $|Q_1|=|Q_2|=q$. Then for any kernel $U\in L^p_{sym}([0,1]^2)$, we have
\begin{align*}
&\max\left\{\max_{\substack{R_1\subseteq Q_1, \\ R_2\subseteq Q_2}} \mc{U}_X(R_{2,X}^+, R_{1,X}^+),\max_{\substack{R_1\subseteq Q_1, \\ R_2\subseteq Q_2}} -\mc{U}_X(R_{2,X}^-,R_{1,X}^-)\right\}\\
\leq & (k-|Q|)^2 \|U\|_\pcut +2k\sum_{i\in Q} U_{X_i}+|\mc{U}_X(Q,Q)|\\&+4(k-1)\sqrt{k\ln k} \|U\|_1 \left(1+ N_ {k,\va,Q_1,Q_2}(X) k^\va\right)\sqrt{ \frac{qN_ {k,\va,Q_1,Q_2}(X) p(\va+3)}{2}}
\end{align*}
for a.e. $X\in [0,1]^k$.
\end{lemma}

\begin{proof}
By Corollary \ref{cor:Azuma_upper}, we have for any $\vb\geq 1$ and $X\in L_{\vb,\va}^0\setminus S_{k,\vb,\va,Q_1,Q_2}$ that
\begin{align*}
\max_{\substack{R_1\subseteq Q_1,\\ R_2\subseteq Q_2}} \mc{U}_X\left(R_{2,X}^*,R_{1,X}^*\right)
\leq & 
\max_{\substack{R_1\subseteq Q_1, \\ R_2\subseteq Q_2}}\E_{X_{[k]\setminus Q}}\left[\mc{U}_X(R_{2,X}^*,R_{1,X}^*)\right]\\
&+
4(k-1)\sqrt{k\ln k} \|U\|_1 \left(1+ \vb k^\va\right)\sqrt{ \frac{q\vb p(\va+3)}{2}}.
\end{align*}
Applying part $(\mc{A})$ of Lemma \ref{le:martingale_diff2} to both $U$ and $-U$, we obtain
\begin{align*}
E_{X_{[k]\setminus Q}}\left[\mc{U}_X(R_{2,X}^*,R_{1,X}^*)\right]\leq&
(k-|Q|)^2  \|U\|^*_\pcut  + 2k \sum\limits_{i \in Q} U_{X_i}+|\mc{U}_X(Q,Q)|\\
\leq&(k-|Q|)^2  \|U\|_\pcut  + 2k \sum\limits_{i \in Q} U_{X_i}+|\mc{U}_X(Q,Q)|.
\end{align*}
Finally, taking the infimum over all valid $\vb$ for a fixed $X$, we obtain the desired claim.
\end{proof}
The expectation of the above upper bound now only depends on $X$ through the random variable $N_{k,\va,Q_1,Q_2}(X)$, whose tail distribution we have control over via \eqref{eq:N_tail}. This can in turn be exploited to bound the expectation of the maximum, a significant step forward compared to the maximum of the expectations.

\begin{prop*}[Proposition \ref{prop:max_expect}]
Let $k\geq 2$, $q\in[k]$ be integers, $p\geq 2$, and $\va>0$. Then for any kernel $U\in L^p_{sym}([0,1]^2)$, we have
\begin{align*}
&\E_X\left[
\E_{\substack{Q_1,Q_2\subseteq [k],\\|Q_1|=|Q_2|=q}}
\left[
\max\left\{\max_{\substack{R_1\subseteq Q_1, \\ R_2\subseteq Q_2}} \mc{U}_X(R_{2,X}^+, R_{1,X}^+),\max_{\substack{R_1\subseteq Q_1, \\ R_2\subseteq Q_2}} -\mc{U}_X(R_{2,X}^-,R_{1,X}^-)\right\}
\right]
\right]\\ \leq&
(k-q)^2\|U\|_\pcut + 4(k+q)q\|U\|_1
+ 4k^\va(k-1)\sqrt{k\ln k} \|U\|_1 \sqrt{ \frac{q p(\va+3)}{2}}
\left(1+k^{1-\va p}\frac{2^{p+3}\|U\|_p^p}{\|U\|_1^p (2p-3)}\right).
\end{align*}
\end{prop*}
\begin{proof}
We have by Lemma \ref{le:max_upper}
\begin{align*}
&\E_X\left[
\E_{\substack{Q_1,Q_2\subseteq [k],\\|Q_1|=|Q_2|=q}}
\left[
\max\left\{\max_{\substack{R_1\subseteq Q_1, \\ R_2\subseteq Q_2}} \mc{U}_X(R_{2,X}^+, R_{1,X}^+),\max_{\substack{R_1\subseteq Q_1, \\ R_2\subseteq Q_2}} -\mc{U}_X(R_{2,X}^-,R_{1,X}^-)\right\}
\right]
\right]\\ \leq&
\E_X\left[
\E_{\substack{Q_1,Q_2\subseteq [k],\\|Q_1|=|Q_2|=q}}
\left[
(k-|Q|)^2 \|U\|_\pcut +2k\sum_{i\in Q} U_{X_i}+|\mc{U}_X(Q,Q)|
\right]\right]\\&+4k^\va(k-1)\sqrt{k\ln k} \|U\|_1 \sqrt{ \frac{q p(\va+3)}{2}}
\E_X\left[
\E_{\substack{Q_1,Q_2\subseteq [k],\\|Q_1|=|Q_2|=q}}
\left[
\left(1+ N_{k,\va,Q_1,Q_2}(X)\right)\sqrt{N_{k,\va,Q_1,Q_2}(X)}
\right]\right]\\
\leq &
(k-q)^2\|U\|_\pcut +4kq\|U\|_1+2q(2q-1) \|U\|_1\\
&+ 4k^\va(k-1)\sqrt{k\ln k} \|U\|_1 \sqrt{ \frac{q p(\va+3)}{2}}
\E_{\substack{Q_1,Q_2\subseteq [k],\\|Q_1|=|Q_2|=q}}
\left[
\E_X\left[
2N^{3/2}_{k,\va,Q_1,Q_2}(X)
\right]\right].
\end{align*}
Now by \eqref{eq:N_tail} we obtain
\begin{align*}
\E_X\left[
N^{3/2}_{k,\va,Q_1,Q_2}(X)
\right]=&\int_0^\infty \mathds{P}\left(N^{3/2}_{k,\va,Q_1,Q_2}(X)\geq s \right)\,ds=
\frac32\int_0^\infty \mathds{P}\left(N_{k,\va,Q_1,Q_2}(X)\geq t \right)\sqrt{t}\,dt\\
\leq& \frac32\int_0^1\sqrt{t}\,dt+\frac32\int_1^\infty \kappa_{p,\va}t^{1/2-p}\,dt=
1+\frac{3\kappa_{p,\va}}{2p-3},
\end{align*}
and so
\begin{align*}
&\E_X\left[
\E_{\substack{Q_1,Q_2\subseteq [k],\\|Q_1|=|Q_2|=q}}
\left[
\max\left\{\max_{\substack{R_1\subseteq Q_1, \\ R_2\subseteq Q_2}} \mc{U}_X(R_{2,X}^+, R_{1,X}^+),\max_{\substack{R_1\subseteq Q_1, \\ R_2\subseteq Q_2}} -\mc{U}_X(R_{2,X}^-,R_{1,X}^-)\right\}
\right]
\right]\\ \leq&
(k-q)^2\|U\|_\pcut + 4(k+q)q\|U\|_1
+ 4k^\va(k-1)\sqrt{k\ln k} \|U\|_1 \sqrt{ \frac{q p(\va+3)}{2}}
\left(1+\frac{3\kappa_{p,\va}}{2p-3}\right).
\end{align*}
\end{proof}

Combining Lemma \ref{Lemma:Frobenius} and Proposition \ref{prop:max_expect} with our upper bound from Corollary \ref{cor:EB}, and using that $kq\leq k\sqrt{qk}$ and $\ln k> 0,8$, we arrive at the following.

\begin{prop}\label{prop:upper_expect}
Let $k\geq 2$, $q\in[k]$ be integers, $p\geq 2$, and $\va>0$. Then for any kernel $U\in L^p_{sym}([0,1]^2)$, we have that
\begin{align*}
 & \E_{X} \left[\|\mc{U}_X\|_\pcut\right]-\|U\|_\pcut\\
& \qquad \leq \|U\|_1k^{\va-1/2+\max\left\{0,1-\va p\right\}}\sqrt{q\ln k}\left\{10
+ 4 \sqrt{ \frac{ p(\va+3)}{2}}
\left(\frac{2^{p+3}\|U\|_p^p}{\|U\|_1^p (2p-3)}\right)
\right\}
       +\frac{2}{\sqrt{q}} \|U\|_2.
\end{align*}
\end{prop}
Now we want to minimize the exponent of $k$ in the second term by choosing $\va$ appropriately, and then pick $q$ so that the last two terms match up, and we obtain the best possible exponent for $k$ in the upper bound. The minimizing choice is $\va=1/p$, whereas for $q$, matching exponents imply $k^{\va-1/2}\sqrt{\ln k}\sim1/q$, leading to $q:=\left\lceil\frac{k^{1/2-\va}}{\sqrt{\ln k}}\right\rceil=\left\lceil\frac{k^{1/2-1/p}}{\sqrt{\ln k}} \right\rceil$.
This then yields Proposition \ref{cor:q_sample}.
\end{document}